\documentclass[11pt]{article}

\usepackage{amsmath,amssymb,amsthm,mathtools,tikz,graphicx}

\newtheorem*{theorem*}{Theorem}
\newtheorem{theorem}{Theorem}[section]
\newtheorem{lemma}[theorem]{Lemma}
\newtheorem{proposition}[theorem]{Proposition}
\newtheorem{corollary}[theorem]{Corollary}

\theoremstyle{definition}
\newtheorem{definition}[theorem]{Definition}
\newtheorem{remark}[theorem]{Remark}

\numberwithin{equation}{section}

\mathtoolsset{showonlyrefs}

\DeclareMathOperator{\supp}{supp}
\DeclareMathOperator{\Tr}{Tr}
\DeclareMathOperator{\Var}{Var}

\title{Global fluctuations for Multiple Orthogonal Polynomial Ensembles}
\author{Maurice Duits\footnote{Department of Mathematics, Royal Institute of Technology (KTH), Stockholm, Swe\-den. Email: duits@kth.se. Supported by the Swedish Research Council grant (VR) Grant no. 2016-05450 and the G\"oran Gustafsson Foundation.} \and Benjamin Fahs \date{}
	\footnote{Department of Mathematics, Imperial College London, London, UK. Email: bf308@ic.ac.uk. Supported by the G\"{o}ran
Gustafsson Foundation (UU/KTH) and by the Leverhulme Trust research programme grant
RPG-2018-260.} \and Rostyslav Kozhan\footnote{Department of Mathematics, Uppsala University, Uppsala, Sweden. Email: kozhan@math.uu.se.}}

\begin{document}
	\maketitle
	\begin{abstract}
		  We study the fluctuations of linear statistics with polynomial test functions for Multiple Orthogonal Polynomial Ensembles.	Multiple Orthogonal Polynomial Ensembles form an important class of determinantal point processes that include random matrix models such as the GUE with external source, complex Wishart matrices, multi-matrix models and others.

		  	Our analysis is based on the recurrence matrix for the multiple orthogonal polynomials, that is constructed out of the nearest neighbor recurrences. If the coefficients for the nearest neighbor recurrences have limits, then we show that the right-limit of this recurrence matrix is a matrix that can be viewed as  representation of a Toeplitz operator with respect to a non-standard basis. This will allow us to prove  Central Limit Theorems for linear statistics of Multiple Orthogonal Polynomial Ensembles. A particular novelty is the use of the  Baker--Campbell--Hausdorff formula to prove that the higher cumulants of the linear statistics converge to zero.
		  	
		  	 We illustrate  the main results by discussing Central Limit Theorems for the Gaussian Unitary Ensembles with external source, complex  Wishart matrices  and specializations of the Schur measure related to multiple Charlier,   multiple Krawtchouk and multiple Meixner polynomials.
	\end{abstract}
	\section{Introduction}
Let $m \in \mathbb N$ and a system of Borel measures $\{\mu_j\}_{j=1}^m$ on
$\mathbb R$ be given.
For each vector $\vec k=(k_1,\ldots, k_m) \in \mathbb Z_{\ge0}^m$, the type II  multiple orthogonal polynomial of multi-index
$\vec k$
is defined as the monic polynomial $p_{\vec k}$ of  smallest possible degree such that
\begin{equation}\label{eq:MOP}
\int p_{\vec k}(x) x^l  \ d \mu_j(x)=0, \qquad \ell=0,\ldots, k_j-1, \quad j=1, \ldots,m.
\end{equation}
We say that the multi-index $\vec k$ is \emph{normal} if the degree of $p_{\vec k}$ is $|\vec k|=k_1+\ldots +k_m$. A system of
measures $\{\mu_j(x)\}_{j=1}^m$ is called  a \emph{perfect system} if  each $\vec k \in \mathbb Z_{\ge0}^m$ is normal. 	Note that if $m=1$ then this definition reduces to that of
standard orthogonal polynomials on the real line.

Multiple orthogonal polynomials arise naturally in
several different contexts. Originally they were introduced in analytic number theory (for proving irrationality and transcendence). They also play an important role in approximation theory (Hermite-Pad\'{e} approximations). We refer to \cite{NS,Ismail} for surveys on these topics. More recently, they have appeared in  random matrix theory and integrable probability,  as they integrate the Multiple Orthogonal Polynomial Ensembles \cite{Kmop}. This is an interesting class of point processes that 
generalize the more classical Orthogonal Polynomial Ensembles \cite{konig} and includes many interesting models such as Unitary Ensembles with external source, complex Wishart matrices, products of random matrices, the two matrix model and certain specializations of the Schur processes.

\subsection{Multiple Orthogonal Polynomial Ensembles}\label{ss:MOPE}

Let us take
\begin{equation}\label{eq:orthomeasures}
d\mu_j(x) = w_j(x) \, d\mu(x), \quad j\in\{1,\ldots,m\}
\end{equation}
for some Borel measure $\mu$ on $\mathbb R$ and non-negative weight functions $w_j$. Fix a vector of non-negative integers $\vec k=(k_1,\ldots,k_m)$ and set $n=k_1+\ldots+k_m=|\vec k|$.  Then a probability measure of the form
\begin{equation} \label{eq:MOPE}
\frac{1}{Z_{\vec{k}}}\det \big( x_i^{j-1}\big)_{i,j=1}^{n}\det\big(g_j( x_i)\big) _{i,j=1}^{n} \prod_{j=1}^{n} {\rm d} \mu(x_j),
\end{equation}
where  $g_j(x)$ are the functions
\begin{equation} \label{eq:MOPEgj}
w_1(x),xw_1(x), \ldots, x^{k_{1}-1} w_1(x), \cdots,w_m(x),xw_m(x), \ldots, x^{k_m-1} w_m(x),
\end{equation}
and $Z_{\vec{k}}$ is a normalizing constant
is called a \emph{Multiple Orthogonal Polynomial Ensemble (MOPE)}.
It should be noted that we assume here that \eqref{eq:MOPE} is non-negative, which is an implicit assumption on the weight functions $w_j$.  In the language of \cite{Bbio} we see that \eqref{eq:MOPE} defines a biorthogonal ensemble, where one of the family of functions consists of polynomials. It is therefore clear that \eqref{eq:MOPE} is a determinantal point process \cite{Jdet}. The name, MOPE, for these processes  was proposed in \cite{Kmop} and is explained by the fact that the correlation kernel can be expressed in terms of multiple orthogonal polynomials. The correlation kernel will however not play a role in the present paper.

A classical example of a MOPE is the Gaussian Unitary Ensemble with external source.  This well-studied model was introduced in \cite{BH}. In the model, we equip the space of $n \times n$ Hermitian matrices with  distribution proportional to
\begin{equation} \label{eq:externalsource}
	\sim e^{-n \Tr \big(\tfrac12 M^2-HM\big)} \ d M,
\end{equation}
where $H$ is a diagonal matrix and $dM$ is the Lebesgue measure on the entries. It turns out \cite{BK,BK2} that the eigenvalues of $M$ form a Multiple Orthogonal Polynomial Ensemble where $m$ is the number of distinct values along the diagonal with multiplicities $k_j$ so that $k_1+\ldots +k_m=n$. The weight functions are the Gaussian
	 $$
	 	w_j(x)=e^{-n\big(\tfrac12 x^2-h_jx\big)},
	 $$
and $\mu$ is the Lebesgue measure on $\mathbb R$.  The polynomials in this case are called multiple Hermite polynomials.

Another example is the complex Wishart Ensemble. Given a diagonal  matrix  $\Sigma=\mathrm{diag}(\sigma_1,\ldots,\sigma_n)$ with $\sigma_j>0$, one studies the eigenvalues of an $n \times n$  matrix $M$ chosen randomly from
\begin{equation}
	\sim	(\det M)^{\alpha} e^{-n \Tr M \Sigma }  \ d M,
\end{equation}
on the space of positive definite hermitian matrices (for $\alpha \in \mathbb N$ the eigenvalues of $M$ are the square of the singular values of a $n \times (n+\alpha)$ where the columns are independent and identically distributed as a complex multivariate normal distribution with   covariance matrix $\Sigma^{-1}$). Similarly to the Gaussian Unitary Ensemble with external source, the eigenvalues of $M$ form a MOPE, where $m$ is the number of distinct values along the diagonal of $\Sigma$ and the $k_j$'s are their multiplicities. The multiple polynomials that arise here are called the multiple Laguerre polynomials of the second kind \cite{BK,Kmop}.

There are many other examples, some of which we address in Section~\ref{sec:examples}. For now, we only briefly mention some important ones and refer to \cite{Kmop} for an extensive survey and a long list of references. In the Hermitian two matrix model, the biorthogonal polynomials  can be characterized as multiple orthogonal polynomials \cite{DKM,KMcL}.  MOPE's also appear in the recent activity around the singular values of products of random matrices \cite{KZ}. The Borodin--Muttalib Ensembles \cite{Bbio,KM} with an integer (or, by duality, the reciprocal of an integer)  parameter $\theta$ lead to multiple orthogonal polynomials with the number of weights $\theta$.

 Lesser known examples are several discrete multiple orthogonal polynomial ensembles that are particular specializations of the Schur measure and arise in the discrete integrable probability models, such as  last passage percolation, discrete interacting particles systems and random tilings of planar domains. Here we briefly mention how the multiple Meixner ensemble arises as a specialization of the Schur measure \cite{OK}. Later in Section \ref{sec:examples} we will also discuss Markov chains that lead to the multiple Charlier, multiple Krawtchouk and multiple Meixner polynomials and  are specializations of the Schur process \cite{OR}.  These are also specializations and marginals  of a more general Markov chain introduced  in  \cite{BF}. In fact, the multiple Charlier model with two weights was already studied  in \cite{Duits1} as an example of the model in \cite{BF},  but the connection to multiple Charlier polynomials was not mentioned there. For completeness, we will include a short discussion on Schur polynomials and the Schur process in Section \ref{sec:Schur}.

  We recall that for $a=(a_1,a_2,\ldots, a_n){\in\mathbb{R}_{\ge0}^n}$ 
  and  a partition $\lambda= (\lambda_1,\lambda_2,\ldots, \lambda_n) {\in \mathbb{Z}_{\ge 0}^n}$ with $\lambda_j \geq \lambda_k$ for $j \leq k$,   the Schur polynomial \cite[I.3]{MacDonald} is defined as
 \begin{equation} \label{eq:Schurpol}
 s_\lambda(a)= \frac{\det (a_{k}^{\lambda_j+n-j})_{j,k=1}^n}{ \det (a_{k}^{n-j})_{j,k=1}^n}.
 \end{equation}
 Note that this is a symmetric multivariate polynomial in the $a_k$'s.
Given $a,b{\in\mathbb{R}_{\ge0}^n}$, we can then define  the Schur measure as the probability measure  on $\lambda= (\lambda_1,\lambda_2,\ldots, \lambda_n)$ with $\lambda_j \geq \lambda_k$ for $j \leq k$,  by
 \begin{equation}\label{eq:preSchurmeasure}
\mathbb P(\lambda)= s_\lambda(a) s_\lambda(b) \prod_{j,k=1}^n\left(1-a_j b_k\right) .
 \end{equation}
 By taking the  limits $a_k\to 1$, $b_k\to q \in (0,1)$ and using l'H\^opital's rule, we see that $s_\lambda(a) s_\lambda(b)$ converges to a product of two Vandermonde determinants in the variables $x_j=\lambda_j-j$ and a factor $q^{x_1+ \ldots + x_n}$ (up to a constant). This is called the Meixner ensemble in the literature since the corresponding orthogonal polynomials that integrate this model are the Meixner polynomials. If, instead, one takes the limits $a_k\to 1$ and $b_k\to q_j$ such that there are precisely $m$ different values  $q_1, \ldots, q_m$ for these limits with multiplicities $k_1,\ldots,k_m$ such that $k_1+\ldots +k_m=n$, then $s_\lambda(a)$ converges again to a Vandermonde determinant but $s_\lambda(b)$ converges, up to a constant, to $\det g_j(x_i)$ where $g_j$ are as in \eqref{eq:MOPEgj}  with $w_j(x)=q_j^x$ and multiplicities $k_j$ (cf. Lemma \ref{lem:schurtoMOP}). Thus, the limiting Schur measure is an MOPE of the form  \eqref{eq:MOPE}. The corresponding multiple orthogonal polynomials are called multiple  Meixner polynomials of the first kind \cite{HV}, see also Section~\ref{ss:Meixner} below.

\subsection{Universality of global fluctuations}

The prime interest of the current paper is to investigate the global behavior of Multiple Orthogonal Polynomial Ensembles. A natural way to study a point process is via its linear statistics. Given a random configuration of points  $\{x_j\}_{j=1}^n$ and a smooth function $f:\mathbb R \to \mathbb R$, the linear statistic $X_n(f)$ is defined as
	$$X_n(f)= \sum_{j=1}^n f(x_j).$$
	
	The first natural question is whether there is a  limiting distribution for the random points $x_j$. In other words, does there exists a $\rho(x)$ such that
	$$
		\frac{1}{n} \mathbb E [X_n(f)] \to  \int f(x) \rho (x) dx,
	$$
	as $n\to \infty$? In many determinantal point processes in random matrix theory and integrable probability this is indeed the case and identifying $\rho(x)$ is an important problem. Naturally, $\rho$  depends on the parameters of the model at hand (in the case of MOPE, these parameters are the weights $w_j$, measure $\mu$, and the indices $k_j$).  It was proved by Hardy in  \cite{Hardy}  that for MOPE's, $\rho(x)$  is the  limiting zero distribution for the multiple orthogonal polynomials $p_n$, and that convergence even holds almost surely (under certain conditions on the recurrence coefficients for the multiple orthogonal polynomials). The characterization of the  limiting zero distribution of $p_n$ is a classical problem for multiple orthogonal polynomials.  This limiting distribution can often be given in terms of the solution of a vector equilibrium problem, which is a good starting point for an asymptotic analysis  of the multiple orthogonal polynomials via Riemann-Hilbert techniques in special cases. We refer to \cite{Kmop,VAGK} for general discussions.

	 In this paper, we will address the phenomenon that the global fluctuations  $X_n(f) -\mathbb E X_n(f)$ are universal. 	One of the remarkable facts, is that for many random processes in random matrix theory the variance of the linear statistic $X_n(f)$ does not grow with $n$ for  functions $f$ that are sufficiently smooth (jump discontinuities are known to contribute with terms that grow logarithmically in $n$). In fact, it has been proved in a variety of models that for sufficiently smooth $f$ there is limit
	$$\lim_{n\to \infty} \Var X_n(f)=\sigma_f^2,$$
	for some $\sigma_f^2>0$.  The models in which this limit holds have the property that the random points accumulate on a single interval. If they accumulate on multiple disjoint intervals, then the variance may have quasi-periodic behavior (see \cite{BD} for a discussion and further references). Note that if the points $x_j$ were to  be independent, then the variance would grow linearly in $n$. That the variance does not grow is a result of the  repulsion between the points that causes rigidity in the location of the points. Moreover, the limiting variance $\sigma_f^2$ is universal and only depends on a few general properties of the model. Even more is true:  for many models, in which the points accumulate on a single interval, it has been established that, as $n \to \infty$,
	\begin{equation}\label{eq:CLT}
	X_n(f)-\mathbb E X_n(f) \to N(0,\sigma_f^2),
	\end{equation}
	in distribution. This Central Limit Theorem (CLT) together with the universality of $\sigma_f$, is  sometimes referred to as global universality and has been proved in a long list of important models. We will not attempt to give a full list of references,  but only mention some. One of the first CLT's in random matrix theory is by Jonsson  who proved a  CLT  for Wishart Matrices \cite{jonsson}. In a seminal work \cite{Jduke}, Johansson proved \eqref{eq:CLT} for  general (continuous) $\beta$-ensembles on $\mathbb R$ with polynomial potentials and certain smoothness conditions on $f$. The case of discrete $\beta$-ensembles was dealt with by Borodin, Gorin and Guionnet \cite{BGG}.  For the special value $\beta=2$, the continuous and discrete log-gases are special cases of Orthogonal Polynomial Ensembles and, consequently, of MOPE's. In \cite{BD} a general CLT for polynomial test functions was proved for a class of biorthogonal ensembles, including Orthogonal Polynomial Ensembles as important examples. The goal of the present paper is to prove a general CLT in \eqref{eq:CLT} for MOPE's, inspired by the approach introduced in \cite{BD}.  An alternative approach to CLT's for MOPE's would be to use the  Schur-generating function approach recently developed  by Bufetov and Gorin \cite{BG}. It would be interesting to see how the techniques and results of the present paper would fit into their framework.

\subsection{Nearest neighbor recurrences}

 The main tool that we will use is the nearest neighbor recurrence relations for multiple orthogonal polynomials.  Similarly to the case of standard orthogonal polynomials on the real line, for each family of multiple orthogonal polynomials (from a perfect system) there exist coefficients $ \{a_{\vec k,j}\}_{\vec k \in \mathbb Z_{\ge0}^m}$ and $\{b_{\vec k,j}\}_{\vec k\in \mathbb Z_{\ge0}^m}$ such that
\begin{equation}\label{eq:nearest_neighbor}
x p_{\vec k}(x) =
p_{\vec k + \vec e_\ell}(x)+ b_{\vec k,\ell} p_{\vec k}(x)+\sum_{j=1}^m a_{\vec k,j} p_{\vec k-\vec e_j}(x)
\end{equation}
holds for every $1\le \ell \le m$, where $\{\vec e_j\}_{j=1}^m$ denotes the standard canonical basis of $\mathbb R^m$. Equations~\eqref{eq:nearest_neighbor} are called the \emph{nearest neighbor recurrence relations}, see~\cite{Ismail, VanAssche11}. If we subtract two of these relations with $\ell=r$ and $\ell=s$, we arrive at the equality

\begin{equation}\label{eq:consistency}
p_{\vec k + \vec e_r}(x)=p_{\vec k + \vec e_s}(x)+ (b_{\vec k,s}-b_{\vec k,r} ) p_{\vec k}(x),
\end{equation}
for $1 \leq  r,s  \leq m$. Note that for $m=1$ these relations are obsolete, so they are a particular feature for multiple orthogonal polynomials.

In concrete examples, the coefficients may be explicitly computed. There are families of multiple orthogonal polynomials corresponding to all of the classical orthogonal polynomials (Hermite, Laguerre, Jacobi,  Charlier,\\  Krawtchouk, Meixner, Hahn), and for these, there are explicit expressions in the literature \cite{ACV,BK,HV,NS,NV, Ismail,VanAssche11}.

In general, the map from the system of measures $\{\mu_j(x)\}_{j=1}^m$ to its nearest neighbor recurrence coefficients is highly non-trivial. There is an important class of measures though for which the coefficients are known to have a simple limiting behavior. \begin{definition}
We say that a perfect system $\{\mu_j(x)\}_{j=1}^m$ is in the  \emph{multiple Nevai class} if, for every $j\in\{1,\ldots,m\}$ and every $\vec \nu=[0,1]^m$ with $|\vec \nu|=1$, there exist $a_j(\vec \nu)$ and $b_j(\vec \nu)$ such that
\begin{equation}\label{eq:multipleNevai}
\begin{cases}
a_{\vec k ,j} \to a_j(\vec \nu),
\\
b_{\vec k,j} \to b_j(\vec \nu),
\end{cases}
\end{equation}
if $\vec k /|\vec k| \to \vec \nu$ as $|\vec  k| \to \infty$.
\end{definition}
As was shown in~\cite[Thm 4.6]{ADY}, multiple Nevai class in particular includes systems for which $\supp\,\mu_j$ are pairwise disjoint intervals (Angelesco systems) 
with $\mu_j$ absolutely-continuous whose a.c. densities $\tfrac{d\mu_j}{dx}$ are non-vanishing and analytic in a neighborhood of $\supp\,\mu_j$. This class is conjectured to be much wider, allegedly containing Angelesco systems with the mere conditions $\tfrac{d\mu_j}{dx}>0$ a.e. on $\supp\,\mu_j$. 
The computation of the constants $\{a_j(\vec \nu)\}_{j=1}^m$ and $\{b_j(\vec \nu)\}_{j=1}^m$ in~\eqref{eq:multipleNevai} was addressed in
\cite{ADY} and
\cite{AptKoz}.

It should be noted that the behavior \eqref{eq:multipleNevai} will typically not hold for measures that have unbounded support. Instead, an alternative formulation of \eqref{eq:multipleNevai} is expected to be true for a wide class of measures with unbounded support. Just as in the case of orthogonal polynomials (e.g., the Hermite polynomials), to get reasonable limiting expressions for the multiple orthogonal polynomials and their features, we need an appropriate rescaling. For example, for the Multiple Hermite Polynomials related to the random matrix model \eqref{eq:externalsource} we have rescaled the Gaussians to be $e^{-n (x-h_j)^2/2}$. This has the consequence that the eigenvalues accumulate on a union of bounded intervals.  By formula (2.2) in \cite{BK},
\begin{equation}\label{eq:multiplehermiterec}
x p_{\vec k}(x)= p_{\vec k+\vec e_\ell}(x)+h_\ell p_{\vec k} (x)+ \sum_{j=1}^m \frac{k_j }{n}p_{\vec k-\vec e_j}(x),
\end{equation}
for $|\vec k| \geq 0$. We see that \eqref{eq:multipleNevai} does not hold directly,  but we do obtain a limit if we let $|\vec k| \to \infty$ and $n \to \infty$ simultaneously.

To accommodate this type of rescaling, we allow the measures to depend on a parameter $n$ and write $p_{\vec k}^{(n)}$, $a_{\vec k,j}^{(n)}$ and $b_{\vec k,\ell}^{(n)}$ to indicate the dependence on $n$. In the orthogonal polynomial literature one often speaks of \emph{varying measures}. Then for a wide class of measures it is expected that the following limits hold
\begin{equation}\label{eq:multipleNevaivaryingPre}
\begin{cases}
a_{\vec k ,j}^{(n)} \to a_j(\vec \nu),
\\
b_{\vec k,j}^{(n)} \to b_j(\vec \nu),
\end{cases}
\end{equation}
if $\vec k /|\vec k| \to \vec \nu$ and $|\vec k|/n \to 1$ as $|\vec  k|,n  \to \infty$. The limits \eqref{eq:multipleNevai} or \eqref{eq:multipleNevaivaryingPre} are easily verified for the generalizations of the classical orthogonal polynomials. We discuss some examples in Section \ref{sec:examples}.

\subsection{Main results}

Our goal is to show that one can formulate a general CLT for MOPE's. In discussing the limiting behavior $X_n(f)$ as $n \to \infty$ we have to clarify how precisely we want to take the limit. Indeed, for each $n$ we need to choose a corresponding vector $\vec k$.  The choice we make is the following: for each $n \in \mathbb Z_{\geq 0}$ we take a vector $\vec k_n \in \mathbb Z_{\geq 0}^m$ such that
\begin{equation}\label{eq:paths}
\left.
\begin{aligned}
& |\vec k_n|=n;  \\
& \vec k_{n+1}= \vec k_n + \vec e_{j_n} \mbox{ for some } j_n \in \{1, \ldots, m\}; \\
& \mbox{there exists a } \vec \nu \in [0,1]^m, |\vec \nu|=1 \mbox{ with } \vec k_n /n \to\vec  \nu \mbox{ as } n \to \infty.
\end{aligned}
\right\}
\end{equation}

One can think of $\{\vec k_n\}_{n=0}^\infty$ as a path on the lattice $\mathbb Z_{\ge0}^m$,
that connects $(0,\ldots,0)$ to infinity in the direction of $\vec \nu$, see Fig.~1 for an example.


\begin{figure}
		\begin{center}
		\begin{tikzpicture}[scale=.8]
\draw[dashed] (0,0) -- (6,8.2);
		 	 \draw[->, thick] (-.5,0) -- (8,0);
		 	 \draw[->, thick] (0,-.5) -- (0,8);
		 	 \draw[help lines] (.5,0) -- (.5,8);
		 	 \draw[help lines] (1,0) -- (1,8);
		 	 \draw[help lines](1.5,0) -- (1.5,8);
		 	   \draw[help lines](2,0) -- (2,8);
		 	  \draw[help lines] (2.5,0) -- (2.5,8);
		 	     \draw[help lines] (3,0) -- (3,8);
		 	  \draw[help lines](3.5,0) -- (3.5,8);
		  \draw[help lines] (4,0) -- (4,8);
		 	   \draw[help lines](4.5,0) -- (4.5,8);
		 	      \draw[help lines] (5,0) -- (5,8);
		 	     \draw[help lines](5.5,0) -- (5.5,8);
		 	     \draw[help lines](6.0,0) -- (6.0,8);
		 	     \draw[help lines](6.5,0) -- (6.5,8);
		 	     \draw[help lines](7.0,0) -- (7.0,8);
		 	     \draw[help lines](7.5,0) -- (7.5,8);

    	 	 \draw[help lines] (0.75,0) -- (0.75,8);
        	 	 \draw[help lines] (0.25,0) -- (0.25,8);
		 	 \draw[help lines] (1.75,0) -- (1.75,8);
		 	 \draw[help lines](1.25,0) -- (1.25,8);
		 	   \draw[help lines](2.75,0) -- (2.75,8);
		 	  \draw[help lines] (2.25,0) -- (2.25,8);
		 	     \draw[help lines] (3.75,0) -- (3.75,8);
		 	  \draw[help lines](3.25,0) -- (3.25,8);
		  \draw[help lines] (4.75,0) -- (4.75,8);
		 	   \draw[help lines](4.25,0) -- (4.25,8);
		 	      \draw[help lines] (5.75,0) -- (5.75,8);
		 	     \draw[help lines](5.25,0) -- (5.25,8);
		 	     \draw[help lines](6.75,0) -- (6.75,8);
		 	     \draw[help lines](6.25,0) -- (6.25,8);
		 	     \draw[help lines](7.75,0) -- (7.75,8);
		 	     \draw[help lines](7.25,0) -- (7.25,8);

		 	   \draw[help lines](0,.5) -- (8,.5);
		 	    \draw[help lines](0,1) -- (8,1);
		 	      \draw[help lines] (0,1.5) -- (8,1.5);
		 	     \draw[help lines] (0,2) -- (8,2);
		 	      \draw[help lines](0,2.5) -- (8,2.5);
		 	     \draw[help lines](0,3) -- (8,3);
		 	    \draw[help lines](0,3.5) -- (8,3.5);
		 	   \draw[help lines](0,4) -- (8,4);
		 	   \draw[help lines] (0,4.5) -- (8,4.5);
		 	   \draw[help lines] (0,5) -- (8,5);
		 	    \draw[help lines] (0,5.5) -- (8,5.5);
		 	    \draw[help lines] (0,6.0) -- (8,6.0);
		 	    \draw[help lines] (0,6.5) -- (8,6.5);
		 	    \draw[help lines] (0,7.0) -- (8,7.0);
		 	    \draw[help lines] (0,7.5) -- (8,7.5);

		 	   \draw[help lines](0,0.75) -- (8,0.75);
		 	   \draw[help lines](0,0.25) -- (8,0.25);
		 	    \draw[help lines](0,1.75) -- (8,1.75);
		 	      \draw[help lines] (0,1.25) -- (8,1.25);
		 	     \draw[help lines] (0,2.75) -- (8,2.75);
		 	      \draw[help lines](0,2.25) -- (8,2.25);
		 	     \draw[help lines](0,3.75) -- (8,3.75);
		 	    \draw[help lines](0,3.25) -- (8,3.25);
		 	   \draw[help lines](0,4.75) -- (8,4.75);
		 	   \draw[help lines] (0,4.25) -- (8,4.25);
		 	   \draw[help lines] (0,5.75) -- (8,5.75);
		 	    \draw[help lines] (0,5.25) -- (8,5.25);
		 	    \draw[help lines] (0,6.75) -- (8,6.75);
		 	    \draw[help lines] (0,6.25) -- (8,6.25);
		 	    \draw[help lines] (0,7.75) -- (8,7.75);
		 	    \draw[help lines] (0,7.25) -- (8,7.25);
		 	
\draw[->,red,very thick] (0,0)--(.5,0)--(.5,.5)--(1.5,.5)--(1.5,0.75)--(3,0.75)--(3,2)--(3.5,2)--(3.5,4)--(4,4)--(4,6)--(4.5,6)--(4.5,6.5)--(5.0,6.5)--
(5.0,6.75)--(5.25,6.75)--(5.25,7.25)--(5.5,7.25)--(5.5,7.5)--(5.75,7.5)--(5.75,8.1);
		 	     \draw (8,-.5) node{$k_1$};
		 	      \draw (-.5,8) node{$k_2$};
	\end{tikzpicture}
	\caption{An example of an up-right path $\{\vec{k}_n\}_{n=0}^\infty$ for $m=2$. As $n \to \infty$ the path will grow to infinity at an angle determined by $\vec \nu = (\tfrac12,\tfrac{\sqrt3}{2})$.}
\end{center}
\end{figure}
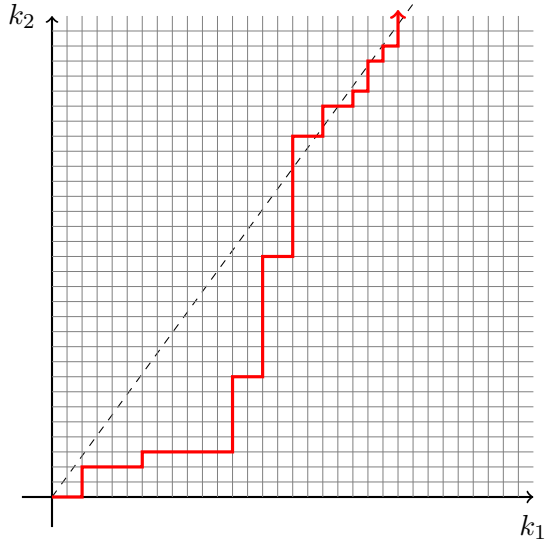

The following two theorems are the main results of the paper. The first one deals with the fixed choice of orthogonality measures and the second allows the measures to vary.

\begin{theorem}[Non-varying systems]\label{MainTheorem}	Let $\mu$ have compact support and assume that the system~\eqref{eq:orthomeasures} belongs to the multiple Nevai class.  Take a path $\{\vec k_n\}_{n=0}^\infty$ satisfying~\eqref{eq:paths}.  Consider $x_1,\ldots,x_n$ taken randomly from \eqref{eq:MOPE} with $m$ weights $w_1,\ldots,w_m$ and multiplicities $\vec k_n$.
	Let $a_j=a_j(\vec \nu)$ and $b_j=b_j(\vec \nu)$ be the limits in \eqref{eq:multipleNevai}. 
	Then, for any polynomial $f $ with real coefficients, we have

	$$\sum_{j=1}^n f(x_j)- \mathbb E\left[ \sum_{j=1}^n f(x_j) \right] \to N\left(0,\sum_{\ell=1}^\infty \ell f_\ell f_{-\ell}\right),$$
	in distribution, where
	$$f_\ell= \frac{1}{2 \pi i} \oint_\gamma f\left(z+ \sum_{j=1}^m\frac{a_j}{z-b_j}  \right) \frac{dz}{z^{\ell +1}}$$
	and $\gamma$ is a contour around the poles $b_j$  with  counter-clockwise orientation.
\end{theorem}

Note that the theorem is essentially a result on the universality of the Central Limit Theorem. This result also presents an additional motivation from integrable probability to study which systems of measures are in the multiple Nevai class.  However, to have a wide scope of applications, the restriction on the compactness of the support of $\mu$ is too severe. Fortunately, it is not necessary.

Let us now allow the weights $w^{(n)}_j$ and the measure $\mu^{(n)}$ in~\eqref{eq:orthomeasures} to be $n$-dependent (where $n$ is the number of points in the process). Then the recurrence coefficients $\{a_{\vec k,j}^{(n)}\}$ and $\{b_{\vec k,j}^{(n)}\}$ are also $n$-dependent. Instead of the Nevai condition~\eqref{eq:multipleNevai} we now only require that for a path \eqref{eq:paths} and for any $j\in\{1,\ldots,m\}$,
\begin{equation}\label{eq:multipleNevaivarying}
\begin{cases}
a_{\vec k_{n+s},j}^{(n)} \to a_j(\vec \nu), \quad \mbox{for all } s\in\mathbb Z,
\\
b_{\vec k_{n+s},j}^{(n)} \to b_j(\vec \nu), \quad \mbox{for all } s\in\mathbb Z,
\end{cases}
\end{equation}
as $n \to \infty$. Note that these right-limit-type conditions are slightly  weaker than \eqref{eq:multipleNevaivaryingPre}.
\begin{theorem}[Varying systems] \label{thm:varying}
    Assume that an $n$-dependent family of perfect systems satisfies~\eqref{eq:multipleNevaivarying} along a path with~\eqref{eq:paths}.
	Then, for any polynomial $f$  with real coefficients, we have
	$$\sum_{j=1}^n f(x_j)- \mathbb E\left[ \sum_{j=1}^n f(x_j) \right] \to N\left(0,\sum_{\ell=1}^\infty \ell  f_\ell f_{-\ell}\right),$$
	in distribution, where
	$$f_\ell= \frac{1}{2 \pi i} \oint_\gamma f\left(z+ \sum_{j=1}^m\frac{a_j}{z-b_j}  \right) \frac{dz}{z^{\ell +1}}$$
	and $\gamma$ is a contour around the poles $b_j$  with  counter-clockwise orientation.
\end{theorem}

Note that Theorem \ref{MainTheorem} is a special case of Theorem \ref{thm:varying}, so it is sufficient to prove Theorem \ref{thm:varying} which will be the topic of the remainder of this paper.

Note that if all $|b_j|<1$   then $f_\ell$ are the Fourier coefficients of the function $f\left(z+ \sum_{j=1}^m\frac{a_j}{z-b_j}  \right) $. Moreover, $f_{-\ell}=\overline{f_\ell}$ and thus the limiting variance is the $H_{\frac12}$-norm of that function. Such $H_{\frac12}$-noise in the fluctuations is typical in random matrix theory. It should also be noted such $H_{\frac12}$-noise can often be seen as a one-dimensional section of a Gaussian Free Field with Dirichlet boundary conditions. Indeed, for certain Markov chains related to orthogonal polynomials, the Gaussian Free Field fluctuations were established in \cite{Duits2}. We believe that it is possible to use similar techniques to prove Gaussian Free Field fluctuations for similar Markov chains involving multiple orthogonal polynomials (for instance, for Dyson's Brownian motion started from several starting points and the Markov chains discussed in Section \ref{sec:examples}). We intend to come back to this problem in future work.

Our approach for proving Theorem \ref{thm:varying} is inspired by  \cite{BD}. In that paper, the authors considered biorthogonal ensembles where the biorthogonal functions satisfy certain recurrence relations. In case that the recurrence coefficients have limits (more precisely, when the recurrence matrix  has a right limit along the diagonal  that is a  banded Toeplitz matrix) then Corollary 2.2 in \cite{BD} says that  there is a CLT of the type \eqref{eq:CLT}.  The special case $m=1$ in Theorem \ref{thm:varying} is  (the varying version of)  Theorem 2.5  in \cite{BD}. The approach of \cite{BD} can also be performed on the unit circle \cite{DK}.  In further works, it has been extended to prove universality of the fluctuations on  the mesoscopic scale for orthogonal polynomial ensembles on the real line \cite{BDmeso} and, very recently, on the circle \cite{BO}.  In \cite{Duits2} a multi-time extension   was constructed to prove the global fluctuations for a class of non-colliding processes (such as stationary Dyson's Brownian motion) are governed by the Gaussian Free Field.  See also  \cite{L} where the results of \cite{BD} were revisited and related to several combinatorial identities.

Although the results of \cite{BD} also hold for some special MOPE's (in fact, \cite{BD} contains  examples including the 2-matrix model and also \cite{L} discusses the  example of singular values of products of Ginibre matrices), it fails to cover important examples such as the ones mentioned previously and which we discuss later in Section \ref{sec:examples}. The results of \cite{BD} assume that there are finite term recurrence relations for the family of polynomials $\{p_{\vec k_n}\}_{n=0}^\infty$. By using the nearest neighbor recurrences \eqref{eq:nearest_neighbor} together with~\eqref{eq:consistency}, we can find the recurrence relations for the family  $\{p_{\vec k_n}\}_{n=0}^\infty$  (cf. Proposition \ref{pr:recurrence}), but the coefficients in these recurrences depend on the nearest neighbor recurrence coefficients in a complicated way. In general, they will \emph{not} have limits which is required in order for the main results of \cite{BD} to apply. Moreover, the number of terms in this recurrence is not bounded and the recurrence matrix is therefore not banded (nor essentially banded as in \cite{BDmeso}). These are the two reasons that our setting falls outside of the scope of \cite{BD} and we  have to come up with new ideas to deal with them.

\subsection{Overview of the paper}

In Section 2 we construct the recurrence  matrix for the  family $\{p_{\vec k_n}\}_{n=0}^\infty$ out of the nearest neighbor recurrences. Moreover, we show that if \eqref{eq:multipleNevai} or \eqref{eq:multipleNevaivarying}  hold, then one can construct a Toeplitz \emph{operator} with a rational symbol out of the limiting values. The recurrence matrix for the  family $\{p_{\vec k_n}\}_{n=0}^\infty$  and the matrix representation of the Toeplitz operator in a particularly chosen basis share the same asymptotic behavior, as shown in Theorem \ref{thm:rightlimitMOPnevai}. It is important to note that this basis is not the canonical basis and the matrix is therefore not a Toeplitz \emph{matrix}.

In Section 3 we discuss certain general determinants involving exponentials of one-sided banded matrices. We show that  the asymptotic behavior  (except for the leading term) of these determinants  depend only on a small part of the one-sided banded matrix. We then prove that under certain conditions, the determinant converges to a Gaussian based on the Baker--Campbell--Hausdorff formula. This last part is an important novelty and replaces the use of Ehrhardt's version of the Helton--Howe--Pincus formula in \cite{BD,BDmeso,Duits2,DK}. The latter is hard to use in our setting since our matrix representation of the Toeplitz operator is not necessarily bounded, leading to non-trivial questions on convergences. By using the Baker--Campbell--Hausdorff formula we avoid such questions and obtain a more direct proof.

In  Section 4 we start by  recalling that the moment-generating function for the linear statistic is a determinant of the type considered in Section 3. After verifying that the condition of the results in Section 3 are satisfied we then arrive at the proof of Theorem \ref{thm:varying}.

Finally, in Section 5 we present some important applications of Theorem \ref{thm:varying}. We show that Theorem \ref{thm:varying} gives CLT's for the Gaussian Unitary Ensemble with external source, complex Wishart matrices and certain Markov processes on non-intersecting paths related to multiple Charlier, multiple Krawtchouk and multiple Meixner polynomials.
\section{Recurrence matrices}
\subsection{Orthogonal polynomials}	
Let us remind the reader the basics of orthogonal polynomials on the real line and of Orthogonal Polynomial Ensembles.
Let $\mu$ be a compactly supported Borel measure on $\mathbb R$ with infinitely many points in its support. The (monic) orthogonal polynomial $p_n$ is the unique monic polynomial of degree $n$ such that
\begin{equation}
\int_{\mathbb R} p_n(x) x^j \ d \mu(x) = 0, \qquad j=0,1,\ldots, n-1.
\end{equation}
One of the important features of orthogonal polynomials on the real line is that they satisfy a three term recurrence. Indeed, there exist $\{a_k\}_{k=0}^{\infty}  \subset (0,\infty)$ and $\{b_k\}_{k=0}^{ \infty}  \subset \mathbb R$  such that
\begin{equation}
x p_k(x)= p_{k+1}(x)+b_k p_k(x)+ a_{k-1} p_{k-1}(x).
\end{equation}
Here we have set $a_{-1}=0$. We can organize this recurrence into a matrix form
\begin{equation}
\label{eq:usualJacobi}
x
\begin{pmatrix}
p_{0}(x)\\
p_{1}(x)\\
p_{2}(x)\\
\vdots
\end{pmatrix}
=
\begin{pmatrix}
b_0 & 1  & \\
a_0 & b_1 & 1 & \\
& a_1 & b_2 & 1\\
&&\ddots & \ddots& \ddots
\end{pmatrix}
\begin{pmatrix}
p_{0}(x)\\
p_{1}(x)\\
p_{2}(x)\\
\vdots
\end{pmatrix}.
\end{equation}
The tridiagonal infinite matrix on the right-hand side of~\eqref{eq:usualJacobi} is called the Jacobi matrix of $\mu$, which we denote by $\mathcal J$.
A significant part of the literature  on orthogonal polynomials is devoted to studying the relationship between (properties of) $\mu$ and (properties of) $\mathcal{J}$ and the coefficients $a_n$'s and $b_n$'s.

An important special class of measures are those that are in the \emph{Nevai class}. This class is defined as all the measures for which there exists $a$ and $b$ such that
\begin{equation}\label{eq:NevaiClass}
a_n\to a, \qquad b_n \to b, \quad \mbox{as } n\to\infty.
\end{equation}
The Nevai class of measures is rather large, but it is not easy to explicitly characterize in terms of properties of $\mu$. A famous result of Denisov--Rakhmanov says that if a measure $\mu$ has essential support $[b-2\sqrt a,b+2\sqrt a]$ and the absolutely continuous part $\tfrac{d\mu}{dx}>0$ almost everywhere on this interval, then $\mu$ is in the Nevai class~\eqref{eq:NevaiClass}.

 We recall that for  $a(z)=\sum_{j} a_jz^j$ the Toeplitz matrix $\mathcal T_a$ with symbol $a$ is defined as
$$\label{defToep} (\mathcal{T}_a)_{jk}=a_{k-j}, \qquad j,k \in \mathbb Z_{\ge0}.$$
Note that the Toeplitz matrix has two standard definitions in the literature, namely the one we have given or its transpose.  In our definition we are motivated by \eqref{formulaToep} below.

Observe that if the measure $\mu$ is in the Nevai class~\eqref{eq:NevaiClass}, then the matrix $\mathcal{J}$ is a compact perturbation of the Toeplitz matrix $\mathcal{T}_s$ with symbol $s(z)=z+b+a/z$. In particular,
\begin{equation}\label{eq:TeoplitzrightlimitOP}
	\mathcal J_{n+j,n+k} -(\mathcal{T}_s)_{n+j,n+k} \to 0,
\end{equation}
as $n \to \infty$.

It is important to note that the orthogonality measures for various classical orthogonal polynomials, such as Hermite and Laguerre polynomials, are not in the Nevai class. Indeed, the measures have unbounded support and the recurrence coefficients are unbounded. This is hardly surprising, since it is known that the asymptotic behavior of the  Hermite and Laguerre polynomials is best described using a rescaling (so that their zeros accumulate on compact intervals), in other words, varying orthogonality. That is, while taking the limit $n \to \infty$ as in \eqref{eq:TeoplitzrightlimitOP},  we introduce the $n$-dependence in the measure $\mu^{(n)}$, the recurrence coefficients $a_k^{(n)}$, $b_k^{(n)}$, monic orthogonal polynomials $p_n^{(n)}$, and the Jacobi matrix $\mathcal J^{(n)}$. If done appropriately, the   following generalization for \eqref{eq:TeoplitzrightlimitOP} is expected to be universal:
\begin{equation}\label{eq:TeoplitzrightlimitOPvarying}
\mathcal J^{(n)}_{n+j,n+k} -(\mathcal{T}_s)_{n+j,n+k} \to 0,
\end{equation}
as $n \to \infty$, for $s(a)=z+b+a/z$ and some $a>0$.

Very roughly speaking, the limit \eqref{eq:TeoplitzrightlimitOPvarying} can be expected to hold if the zeros of $p_n^{(n)}$ accumulate on a single interval $[b-2\sqrt a,b+2 \sqrt a]$. If they accumulate on several intervals, the limit is not true. This  has been proved for varying measures of the type $e^{-n V(x)} dx$ where $V$ is real analytic function with sufficient growth at infinity \cite{DKMVZ}.

\subsection{Recurrence matrices for multiple orthogonal polynomials}\label{sec:J}

We now return to
multiple orthogonal polynomials: let $\{\mu_j\}_{j=1}^m$ 
be a perfect system of measures.
Then 
the type II multiple orthogonal polynomials satisfy the nearest neighbor recurrence relations~\eqref{eq:nearest_neighbor} and the relations \eqref{eq:consistency}.

Now let us choose any path with~\eqref{eq:paths}.
Note that for each $n \in \mathbb Z_{\ge0}$ we have that $x p_{\vec k_n}(x)$ is a polynomial of degree $n+1$. Therefore there exist coefficients $J_{n,j}$ such that
$$x p_{\vec k_n}(x)= p_{\vec k_{n+1}}(x)+ \sum_{j=0}^n J_{n,j} p_{\vec k_{j}}(x).$$
These coefficients can be organized into a matrix $J= (J_{n,j})_{n,j=0}^\infty$ (we take $J_{n,n+1}=1$ for all $n$ and $J_{n,j}=0$ if $j\ge n+2$), and the above relation can be written as
\begin{equation}
\label{eq:generalJacobi}
x
\begin{pmatrix}
p_{\vec k_0}(x)\\
p_{\vec k_1}(x)\\
p_{\vec k_2}(x)\\
\vdots
\end{pmatrix}
=
J  \begin{pmatrix}
p_{\vec k_0}(x)\\
p_{\vec k_1}(x)\\
p_{\vec k_2}(x)\\
\vdots
\end{pmatrix}.
\end{equation}
The matrix $J$ with $J_{n,j}=0$ for $j\ge n+2$ is said to be in (lower) Hessenberg form. The lower triangular part of $J$ can be computed from the nearest neighbor recurrence \eqref{eq:nearest_neighbor} and \eqref{eq:consistency}, as we show in the next proposition.

\begin{proposition}\label{pr:recurrence}
	For each $n\in\mathbb Z_{\geq 0}$, let $j_n\in\{1,\ldots,m\}$ be as in~\eqref{eq:paths}.
	
	Then
	\begin{multline}\label{eq:coefJ}
	x p_{\vec k_n}(x)
	= p_{\vec k_{n+1}}(x)+ b_{\vec k_n, j_n} p_{\vec k_n}(x)+\left( \sum_{\ell=1}^m a_{\vec  k_n,\ell} \right) p_{\vec k_{n-1}}(x)\\
	+\sum_{r=1}^{n-1}   \left( \sum_{\ell=1}^m a_{\vec  k_n,\ell}  B_{n,\ell}B_{n-1,\ell}  \cdots B_{n-r+1,\ell} \right) p_{\vec k_{n-r-1}}(x),
	\end{multline}
	where
	\begin{equation}\label{eq:coefficientsBm}
	B_{n,\ell}= b_{\vec k_{n-1}-\vec e_\ell,\ell} - b_{\vec k_{n-1}-\vec e_\ell,j_{n-1}}.
	\end{equation}
\end{proposition}

\begin{proof}
	From \eqref{eq:consistency} we find
	$$
	p_{\vec k_n-\vec e_\ell}=p_{\vec k_{n-1}}+ B_{n,\ell} p_{\vec k_{n-1}-\vec e_\ell},
	$$
	and thus
	\begin{multline}
	x p_{\vec k_n}(x)= p_{\vec k_{n+1}}(x)+ b_{\vec k_n, j_{n}} p_{\vec k_n}(x)
	+ \sum_{\ell=1}^m a_{\vec k_n,\ell} p_{\vec k_n-\vec e_\ell}\\
	= p_{\vec k_{n+1}}(x)+ b_{\vec k_n, j_n} p_{\vec k_n}(x)
	+ \left( \sum_{\ell=1}^m a_{\vec  k_n,\ell} \right) p_{\vec k_{n-1}}(x)	\\
	+ \sum_{\ell=1}^m a_{\vec  k_n,\ell}  B_{n,\ell}  p_{\vec k_{n-1}-\vec e_{\ell} }(x)\\
	= p_{\vec k_{n+1}}(x)+ b_{\vec k_n, j_n} p_{\vec k_n}(x)+\left( \sum_{\ell=1}^m a_{\vec  k_n,\ell} \right) p_{\vec k_{n-1}}(x)\\
	+\sum_{r=1}^{n-1}   \left( \sum_{\ell=1}^m a_{\vec  k_n,\ell}  B_{n,\ell}B_{n-1,\ell}  \cdots B_{n-r+1,\ell} \right) p_{\vec k_{n-r-1}}(x).
	\end{multline}
	This proves the statement.
\end{proof}
We see that $J$ is more complicated than the Jacobi matrix $\mathcal J$ for orthogonal polynomials. It is natural to ask whether  \eqref{eq:multipleNevaivaryingPre} implies that there exists a Laurent polynomial $s(z)=\sum_{j=-q}^p s_jz^j$ such that
	\begin{equation}\label{eq:TeoplitzrightlimitMOPvarying}
J^{(n)}_{n+j,n+k} -(\mathcal{T}_s)_{n+j,n+k} \to 0,
\end{equation}
as $n \to \infty$, similarly to \eqref{eq:TeoplitzrightlimitOPvarying}. If \eqref{eq:TeoplitzrightlimitMOPvarying}  holds, then Theorem \ref{thm:varying} follows from the results in \cite{BD} and we are done.   Although it is true for some special cases (and we refer to  \cite{BD} for  examples), in general it is too much to ask for.

 First of all, note that $J$   depends on the specific path $\{k_{\vec n}\}$ and, in general, $J$ will not be a banded matrix but only one-sided banded (Hessenberg). The matrix $J$ will happen to be banded in very special cases  such as the step-line path $\vec k_n = \vec{k}_{n-1} + \vec e_{n\!\mod m}$ (here $n\!\mod m$ is the remainder of the Euclidean division of $n$ by $m$), which results in the so-called \emph{step-line recurrence relation}.

 Secondly, 	  $J$ depends on the nearest neighbor recurrences in a complicated fashion and there is no reason to expect that the entries of  $J$ behave in a simple way, even if \eqref{eq:multipleNevaivarying} holds.  In the examples in Section \ref{sec:examples} we will see cases where one can obtain perturbations of block Toeplitz matrices, which is likely as good as it gets.

 Our first main result of this paper is that \eqref{eq:TeoplitzrightlimitMOPvarying} does hold if $\mathcal{T}_s$ is replaced by the matrix representation of a Toeplitz operator in a \emph{non-standard basis}, which we  discuss in the next subsection.

		\subsection{Limiting behavior of recurrence matrices}
		
	Denote the space of rational functions on $\mathbb C$ by $\mathcal R$ and   the (sub-)space of all polynomials by $\mathcal P$. We denote the operator that projects any rational function to its polynomial part by $P: \mathcal R \to \mathcal P$  and denote the embedding of $\mathcal P$ into $\mathcal R$ by $P^*: \mathcal P \to \mathcal R$.  Note that
	$$Pg(z)= \frac{1}{2 \pi i } \oint_\gamma \frac{g(w) dw }{w-z}$$
	where $\gamma$ is a simple counter-clockwise oriented contour that goes around all poles of $g$ and around the pole at $z$, and that with this integral representation it is easily verified that $Pp=p$ for any polynomial $p$ and that $P(1/p)=0$ if the degree of $p$ is 1 or greater.   Then for any rational function $r$ we define the multiplication operator $M_r :\mathcal R \to \mathcal R$ by $M_r g=r g$ for all $g \in\mathcal R$.  The Toeplitz operator $\tau_r: \mathcal P \to \mathcal P$, with symbol $r\in \mathcal R$, is now defined as the operator  $$\tau_r= P M_r  P^*.$$
We will be specifically interested in the Toeplitz operator $\tau_c$   with the symbol
\begin{equation}\label{eq:symbol}
c(z)=z+ \sum_{j=1}^m \frac{a_{j} }{z-b_j}.
\end{equation}
with $m\in \mathbb N$, $b_1, \ldots, b_m \in \mathbb R$ and $a_1, \ldots a_m \in \mathbb R$.
		Thus, for any polynomial $p$,
		$$\left(\tau_c p\right)(z)=
		\frac{1}{2 \pi i} \oint_{|w|=\alpha} p(w) \left(w+\sum_{j=1}^m \frac{a_j}{w-b_j}\right) \frac{dw}{w-z},$$
		where 
		$ \alpha> \max \{|b_1|, \ldots, |b_m|,|z|\}.$

	 For each $\vec k = (k_1,\ldots,k_m)\in \mathbb Z_{\geq 0}^m$ we define
 			$$\pi_{\vec k}(z)= \prod_{j=1}^m (z-b_j)^{k_j}.$$
	 Note that $\tau_c$ acts on such functions in a rather simple way. In fact, we have the following relation
\begin{equation}\label{eq:identities1}
\tau_c \pi_{\vec k} = \pi_{\vec k + \vec e_\ell}+ b_\ell \pi_{\vec k} +\sum_{j=1}^m a_j \pi_{\vec k-\vec e_j},
\end{equation}
	for $\ell=1,\ldots, m$ and $\vec k \in \mathbb Z_{\geq 0}^m.$ By subtracting any pair of equations we also find
\begin{equation}\label{eq:identities2}
\pi_{\vec k+ \vec e_j}= \pi_{\vec k + \vec e_\ell}+ (b_\ell-b_j) \pi_{\vec k}
\end{equation}
	for $\ell,j=1,\ldots,m$. It is illuminating to compare these two identities with~\eqref{eq:nearest_neighbor} and~\eqref{eq:consistency}
for multiple orthogonal polynomials.
	
	The collection of all functions $\pi_{\vec k}$ can not be a basis since it  has too many elements. We need to choose one polynomial of degree $n$ for each $n\in\mathbb Z_{\geq 0}$, to obtain a basis. 	We do this using the path $\{\vec k_n\}_{n=0}^\infty$ with~\eqref{eq:paths}.
	
	In analogy with~\eqref{eq:generalJacobi}, we define $T_c=\left(T_c\right)_{j,k=0}^\infty $  as follows
	\begin{equation}\label{eq:matrixTc}
	\begin{pmatrix} (\tau_c \pi_{\vec k_0})(x)\\ (\tau_c \pi_{\vec k_1})(x)\\ (\tau_c \pi_{\vec k_2})(x)\\ \vdots\end{pmatrix}
	=T_c\begin{pmatrix}\pi_{\vec k_0}(x)\\ \pi_{\vec k_1}(x)\\ \pi_{\vec k_2}(x)\\ \vdots\end{pmatrix}.
	\end{equation}
		
	\begin{theorem}\label{thm:rightlimitMOPnevai}
    Assume that an $n$-dependent family of perfect systems  satisfies~\eqref{eq:multipleNevaivarying} along a path with~\eqref{eq:paths}.
		Let $T_c$ be as in~\eqref{eq:matrixTc} and $J^{(n)}$ be the lower Hessenberg matrix corresponding to multiplication by $x$  in the basis $\{p^{(n)}_{\vec k_n}(x)\}_{n=0}^\infty$ (see~\eqref{eq:generalJacobi}). Then, as $n \to \infty$,
		\begin{equation}\label{eqnRightLim}
			(J^{(n)})_{n+s,n+r}-\left(T_c \right)_{n+s,n+r}\to 0, \quad  \forall s,r\in \mathbb Z,
		\end{equation}
		where $c$ is~\eqref{eq:symbol} with $a_\ell= a_{\ell}(\vec \nu)$ and  $b_{\ell}= b_{\ell}(\vec \nu)$.		
	\end{theorem}
\begin{proof}
	Repeating the proof of Proposition~\ref{pr:recurrence} with~\eqref{eq:identities1}/\eqref{eq:identities2} instead of~\eqref{eq:nearest_neighbor}/\eqref{eq:consistency}, we arrive at
\begin{multline}
	\tau_c \pi_{\vec k_n} = \pi_{\vec k_{n+1}}+ b_{j_n} \pi_{\vec k_n}
    +\left( \sum_{\ell=1}^m a_{\ell} \right) \pi_{\vec k_{n-1}}(x)\\
	+\sum_{r=1}^{n-1}   \left( \sum_{\ell=1}^m a_{\ell}  \tilde{B}_{n,\ell}\tilde{B}_{n-1,\ell}  \cdots \tilde{B}_{n-r+1,\ell} \right) \pi_{\vec k_{n-r-1}}(x),
	\end{multline}
	where
	\begin{equation}\label{eq:coefficientstildeBm}
	\tilde{B}_{n,\ell}= b_{\ell} - b_{j_{n-1}}.
	\end{equation}

Using this and~\eqref{eq:coefJ} (with $n$-dependence), for $s \ge r+2$, we get
\begin{multline*}
(J^{(n)})_{n+s,n+r}-\left(T_c \right)_{n+s,n+r}
=
\left( \sum_{\ell=1}^m a^{(n)}_{\vec  k_{n+s},\ell}  B^{(n)}_{n+s,\ell}B^{(n)}_{n+s-1,\ell}  \cdots B^{(n)}_{n+r+2,\ell} \right)
\\
-
\left( \sum_{\ell=1}^m a_{\ell}  \tilde{B}_{n+s,\ell}\tilde{B}_{n+s-1,\ell}  \cdots \tilde{B}_{n+r+2,\ell} \right)
\end{multline*}
Now the condition~\eqref{eq:multipleNevaivarying} implies that the last expression goes to $0$. For $s=r$ and $s=r+1$ the proof is similar, while for $s<r$ there is nothing to prove due to the Hessenberg structure of $J^{(n)}$ and $T_c$.
\end{proof}

Note that the operator $\tau_c$ is a Toeplitz \emph{operator}, but $T_c$ is \emph{not} a Toeplitz \emph{matrix}. To get a Toeplitz matrix,
we change the basis $\{\pi_j(z)\}_{j=0}^\infty$ to  $\{z^j\}_{j=0}^\infty$  in the equality~\eqref{eq:matrixTc}. We obtain
\begin{equation}
\label{formulaToep}
\begin{pmatrix}\tau_c (1)\\ \tau_c (z) \\ \tau_c (z^2) \\ \vdots \end{pmatrix}
=\mathcal T_c \begin{pmatrix}1\\z\\z^2\\ \vdots \end{pmatrix},
\end{equation}
where the matrix $\mathcal T_c$ is the Toeplitz matrix (see~\eqref{defToep}) with the symbol $c(z)$~\eqref{eq:symbol}.
 The matrices $\mathcal{T}_c$ and $T_c$ are related as indicated in the following lemma.

\begin{lemma}\label{S-1TS}
For any $r \in \mathcal R$ we have
$$T_r=S \mathcal{T}_r S^{-1},$$
where
$S$ is the triangular matrix
$$ S_{j,k}= \frac{1}{2 \pi i} \oint_{\gamma} \pi_{\vec k_{j}}(z) \frac{d z}{z^{k+1}}, \quad j,k=0,1,2, \ldots,$$
where $\gamma_0$ is simple counter-clockwise oriented contour that goes around the origin. 
\end{lemma}
\begin{proof}
	By the Cauchy residue theorem,
	\begin{equation}
	S\begin{pmatrix}1\\z\\z^2\\ \vdots \end{pmatrix}= \begin{pmatrix} \pi_0(z)\\ \pi_1(z)\\ \pi_2(z)\\ \vdots \end{pmatrix}.
	\end{equation}
	The lemma follows from the fact that $\tau_r$ is linear on the space of polynomials~$\mathcal P$.
	\end{proof}

	\section{One-sided band matrices}\label{sec:one-sided}
	In this section, we recall and extend parts of \cite{BD} that are necessary for our proofs. The main results are Theorems \ref{lemmacumu} and \ref{thm:decompB}. The most important differences are: (i) the analysis of \cite{BD} is restricted to band matrices, where as here we only require the matrices to be banded from one side; (ii) we will use the Baker--Campbell--Hausdorff formula to establish that the fluctuations are Gaussian in a more direct way than in \cite{BD}.

\subsection{Determinants and traces of one-sided banded matrices}
	Consider the space $\mathcal B$ of one-sided banded matrices $B= (B_{i,j})_{i,j=1}^\infty$. More precisely, $B \in \mathcal B$ if and only if  there exists  $b>0$ (depending on $B$) such that $B_{ij}=0$ for $j > i+b$.  It is not hard to see that the space of such one-sided band matrices is closed  under  addition and matrix multiplication. Therefore, if $B \in \mathcal B$ then  for any polynomial $p$ the matrix $p(B)$  is well-defined and belongs to $\mathcal B$. In particular,
	$$\exp_r( B):= \sum_{j=0}^r \frac{B^j}{j!}$$
	is well defined.
	
	For $t \in \mathbb R$ and $B \in \mathcal B$, we will be interested in
	$$\det\left( P_n \exp_r (t B) P_n +Q_n\right),$$
	where $Q_n$ and $P_n$ are the complementary projection matrices: $P_n$ is the projection $(x_0, x_1,x_2,\ldots)^T\mapsto(x_0,x_1, \ldots,x_{n-1},0,0, \ldots)^T$, and $Q_n=I-P_n$.
\begin{lemma} \label{lem:cumutraces} For $B \in \mathcal B$, define
		\begin{equation}
		C_m^{(n)}( B)=m!\sum_{j=1}^m \frac{(-1)^{j+1}}{j}\sum_{\substack{l_1+\dots+l_j=m\\ l_i\geq 1}}\frac{{\rm Tr} P_n B^{l_1}P_n\dots  B^{l_j} P_n}{l_1!\dots l_j!},
		\end{equation}
		Then
		\begin{equation}
\log	\det\left( P_n \exp_r (t B) P_n +Q_n\right)= \sum_{m=1}^r \frac{t^m}{m!} C_m^{(n)}(B)+\mathcal O(t^{r+1}),
\end{equation}
as $t \to 0$.
	\end{lemma}
	\begin{proof} This is a standard identity that can be proved by writing
		\begin{multline*}
\log	\det\left( P_n \exp_r (t B) P_n +Q_n\right)= \log \det (I+ P_n( \exp_r (t B) -I)P_n)\\= \Tr \log   (I+ P_n( \exp_r (t B) -I)P_n).
		\end{multline*}
		Expanding the logarithm and then the exponential near $t=0$, we get the statement.  See, for example, \cite[Lemma 4.1]{Duits2} (with $N=1$).
	\end{proof}

 	The following theorem is an extension of the results from \cite{BD} for (two-sided) banded matrices.
	\begin{theorem}\label{lemmacumu}
		Let $B^{(1)}, B^{(2)} \in \mathcal B$ and suppose that
		\begin{equation}\label{Rightlimassum}
		\left|B^{(1)}_{n+i,n+j}-B^{(2)}_{n+i,n+j}\right|\to 0,
		\end{equation}
		as $n\to \infty$, for any fixed $i,j\in \{1,2,\dots\}$. Then, as $n\to \infty$,
		\begin{equation}\label{limCmn}
		\left| C_m^{(n)}\left(B^{(1)}\right)-C_m^{(n)}\left(B^{(2)}\right)\right| \to 0, \end{equation}
		as $n\to \infty$, for any fixed $m=2,3,\dots$.
	\end{theorem}
	\begin{proof}
		We first recall the statement of Lemma 4.1 in \cite{BD}. In equality $t=\log(1+(e^t-1))$ we expand the logarithm first and then the exponential, and then observing that all coefficients of $t^m$ in the expansion vanish, we get
		\begin{equation}
		\label{eq:trivial}
		\sum_{j=1}^m \frac{(-1)^{j+1}}{j} \sum_{\substack{\ell_1+ \ldots+\ell_j=m\\ \ell_i \geq 1}}   \frac{1}{{\ell_{1}! \cdots \ell_{j}!}}=0.
		\end{equation}
		for $m \geq 2$. As a consequence,	it follows that for any infinite matrix $B=(B_{ij})_{i,j=1}^\infty$,  and $m\geq 2$,
		\begin{equation}\label{Cumu1}
		C_m^{(n)}(B)=m!\sum_{j=2}^m \frac{(-1)^{j+1}}{j} \sum_{\substack{l_1+\dots+l_j=m\\ l_i\geq 1}}\frac{{\rm Tr}P_nB^{l_1}P_n\dots B^{l_j}P_n-{\rm Tr} P_nB^m P_n}{l_1!\dots l_j!},
		\end{equation}
		which is Lemma 4.1 in \cite{BD}.
		
        Let us now fix $m\ge 2$. For $B \in\mathcal B$ with $B_{ij}=0$ if $j>i+b$, we define 
\begin{equation}\nonumber
		\widehat B^{(n)}_{ij}=\begin{cases}
		B_{ij}&{\rm for }\,\, |i-n| \textrm{ and } |j-n|<2 m^2 b,\\
		0&{\rm else.}
		\end{cases}\end{equation}
		If we can show that
		\begin{equation}\label{CmnhatB}
		C_m^{(n)}(B)=C_m^{(n)}\left(\widehat B^{(n)}\right),
		\end{equation}
		for $n$ sufficiently large,
		then, by \eqref{Rightlimassum} and \eqref{CmnhatB}, we obtain \eqref{limCmn} and this finishes the proof.
		
		To prove \eqref{CmnhatB}, we use the representation \eqref{Cumu1}, and analyze each summand on the right hand side.
		By definition of the trace, we have
		\begin{multline}\label{eq:bigSum}
		{\rm Tr} P_nB^{l_1}P_n \dots B^{l_j}P_n-{\rm Tr} P_nB^mP_n\\=-\sum_{s=1}^n\sum_{(r_1,\dots,r_{j-1})\in I_{j,n}}\left(B^{l_1}\right)_{sr_1}\left(B^{l_2}\right)_{r_1r_2}\dots \left(B^{l_j}\right)_{r_{j-1}s},
		\end{multline}
		where
		\begin{equation}\nonumber
		I_{j,n}=\left\{(r_1,\dots,r_{j-1}): r_i>n \textrm{ for some }i=1,\dots,j-1\right\}.
		\end{equation}
        Since $B$ is one-sided banded, it follows that
		\begin{equation} \label{Bli}\left(B^{l_i}\right)_{r_{i-1}r_i}=0, \qquad i=1,\dots, j, \end{equation}
		if $r_{i}>r_{i-1}+l_i b$, with $r_0=r_j=s$. Since $l_i\leq m$, it follows that \eqref{Bli} is  $0$ if $r_{i}>r_{i-1}+b m$. Combined with the fact that $s=r_0 \leq n$ and that there is an $i$ such that $r_i>n$, it follows that the only possible contributions in~\eqref{eq:bigSum} can come from the summands where $|r_i-n|<j  m b\leq m^2b$ for all $i$ (including $r_0=r_j=s$). Finally, for a chosen $i=0,\ldots,j$, let us write
$$
(B^{l_i})_{r_{i-1} r_i} = \sum_{(t_1,t_2,\ldots,t_{l_{i}-1})} B_{r_{i-1} t_1} B_{t_1 t_2} \cdots B_{t_{l_i-1}  r_i}.
$$
Note that each of these summands is zero unless $t_k \le t_{k-1} + b$ (we adopt $t_0 = r_{i-1}$ and $t_{l_i} = r_i$). Combining it with $|r_i-n|< m^2b$, this leads to $|t_k-n|\le m^2 b + l_i b < 2m^2 b$.

 Therefore $C_m^{(n)}(B)$ only depends on the entries $B_{ik}$ of $B$ for which $|i-n|,|k-n|<2 m^2b$. This proves~\eqref{CmnhatB}.
	\end{proof}

The following is an important improvement over \cite{BD,DK}.
	
\begin{theorem}\label{thm:decompB}
	Let $B \in \mathcal B$ and decompose $B= B_-+B_+$ such that
	\begin{enumerate}
        \item $B_-$ is strictly lower triangular;
		\item $B_+$ is upper triangular and banded, that is, there exists  $b \in \mathbb N$ such that $(B_+)_{j, j+\ell}=0$ for $\ell<0$ and $\ell >b$;
		\item $[B_-,B_+]Q_s= (B_-B_+-B_+B_-)Q_s= 0$ for some $s\in \mathbb N$.
 		\end{enumerate}
	Then, 	for $m\geq 3$,
	$C_m^{(n)}(B) \to 0 $  as $n \to \infty$.
	\end{theorem}

The proof will make use of the Baker--Campbell--Hausdorff formula. We postpone the proof for a moment and proceed with an intermezzo on several well-known identities around the Baker--Campbell--Hausdorff formula and some of their consequences.

\subsection{Baker--Campbell--Hausdorff formula}
	The Baker--Campbell--Hausdorff formula gives an expression of $Z$ in $\exp Z=\exp X \exp Y$ in terms of nested commutators of $X$ and $Y$.

Before we state the form that we will need, we start with a more straightforward expansion stated in the next lemma.
 \begin{lemma}
 Let $A_1,A_2 \in \mathcal B$, then

	\begin{equation}
		\label{eq:preBCH}
				\log e^{t A_1} e^{tA_2}= \sum_{m=1}^\infty t^m  \sum_{j=1}^m \frac{(-1)^{j+1}}{j} \sum_{\substack{u_1+v_1+\dots+u_j+v_j=m\\ u_i+v_i \geq 1}}  \frac{A_1^{u_1} A_2^{v_1}\cdots A_1^{u_j} A_2^{v_j}}{u_1!v_1 ! \cdots u_j! v_j!}.
	\end{equation} \end{lemma}

	\begin{remark}	The series on the right-hand side \eqref{eq:preBCH} should be considered as formal series. As noted in the previous subsection, the space $\mathcal B$ forms an algebra, so each of the coefficients on the right-hand side of \eqref{eq:preBCH} is well defined. Therefore the statement is to be interpreted in the following standard way: if, on the left-hand side, the exponential is replaced by its truncated series $\exp_r$ and similarly for the logarithm, then each coefficient on the left-hand side matches the corresponding coefficient on the right-hand side for sufficiently large $r$.
		\end{remark}
			\begin{proof}
		This follows by expanding first the logarithm and then the exponentials in the following sequence of equalities:
		\begin{multline}
		\log e^{t   A_1 }e^{t   A_2 }
		= \sum_{j=1}^\infty \frac{(-1)^{j+1}}{j}  (e^{t  A_1}e^{t  A_2}-I)^j\\
		=   \sum_{j=1}^\infty \frac{(-1)^{j+1}}{j}  	\sum_{\substack{u_1,v_1, \ldots, u_j,v_j=0 \\ u_i+v_i\geq 1}}^{\infty} t^{u_1+v_1+ \ldots u_j+v_j}  \frac{A_1^{u_1} A_2^{v_1}\cdots A_1^{u_j} A_2^{v_j}}{u_1!v_1!\cdots u_j!v_j!}
		\end{multline}
		
		The sum can be further reorganized to
		\begin{multline}
\log e^{t   A_1 }e^{t   A_2 }  =   \sum_{j=1}^\infty \frac{(-1)^{j+1}}{j}  \sum_{m=j}^\infty t^m  	 \sum_{\substack{u_1+v_1+ \ldots +u_j+v_j=m \\ u_i+v_i\geq 1}}  \frac{A_1^{u_1} A_2^{v_1}\cdots A_1^{u_j} A_2^{v_j}}{u_1!v_1!\cdots u_j!v_j!}\\
=  \sum_{m=1}^\infty  t^m  \sum_{j=1}^m\frac{(-1)^{j+1}}{j}  \sum_{\substack{u_1+v_1+ \ldots +u_j+v_j=m \\ u_i+v_i\geq 1}}  \frac{A_1^{u_1} A_2^{v_1}\cdots A_1^{u_j} A_2^{v_j}}{u_1!v_1!\cdots u_j!v_j!},
\end{multline}
		which gives the statement.
		\end{proof}
	
This immediately implies the following.

\begin{corollary} \label{cor:trexpa}
Let  $A_1$ and $A_2$ be infinite matrices with only finitely many non-zero entries. Then
	$$ \sum_{j=1}^m \frac{(-1)^{j+1}}{j} \sum_{\substack{u_1+v_1+ \ldots +u_j+v_j=m \\ u_i+v_i\geq 1}}  {\rm Tr}\frac{A_1^{u_1} A_2^{v_1}\cdots A_1^{u_j} A_2^{v_j}}{u_1!v_1!\cdots u_j!v_j!}=0,$$
	for $m\geq 2$.
\end{corollary}
\begin{proof} 
	Since $A_1$ and $A_2$ have only finitely many entries, we can think of them as finite matrix. This means that we can write  
	\begin{multline*}
	  \Tr \log e^{tA_1} e^{tA_2}=\log \det \left( e^{t A_1} e^{t A_2} \right)=\log \det \left( e^{t A_1}\right) \det \left( e^{t A_2} \right)\\=
	  \log \left( e^{t \Tr A_1} e^{t \Tr A_2}\right)= t \Tr A_1 +t \Tr A_2.
	\end{multline*}
	Thus, the trace of the left-hand side of \eqref{eq:preBCH} is linear in $t$ for finite matrices. Therefore the trace of the  coefficients of $t^m$ for $m \geq 2$ on the right-hand side must vanish and  this proves the statement.
\end{proof}

\begin{remark}
	By taking $A_2=0$, the equality reduces to \eqref{eq:trivial}.
\end{remark}
\begin{remark}
	It is important to note that Corollary  \ref{cor:trexpa} does not hold for general infinite matrices $A_1$ and $A_2$. Indeed, if $A_1$ or $A_2$ is not of trace class the traces in the summand are not necessarily well-defined and the statement does not make sense.
	\end{remark}

The Baker--Campbell--Hausdorff formula now states the far from obvious results that each coefficient of $t^m$ on the right-hand side of \eqref{eq:preBCH} can be written as a nested commutator as stated in the next proposition. The form that we have chosen here is due to Dynkin~\cite{Dynkin}.

Let $X_1,X_2,\dots,X_j$ be in $\mathcal{B}$. 
We define the nested commutator by

\begin{equation}
[X_1,X_2,\dots,X_j]=[X_1,[X_2,[\dots,[X_{j-2},[X_{j-1},X_j]]\dots]]. \end{equation}
For example,
\begin{equation}
[X_1,X_2,X_3]=[X_1,[X_2,X_3]]. \end{equation}
We can also take ``powers'' within nested commutator:
\begin{equation}
[X_1^{(\ell_1)},X_2^{(\ell_2)},\dots, X_j^{(\ell_j)}]=[\underbrace{X_1,X_1,\dots,X_1}_{\ell_1},\underbrace{X_2,X_2,\dots,X_2}_{\ell_2},\dots,\underbrace{X_j,X_j,\dots,X_j}_{\ell_j}].
\end{equation}
For example,
\begin{equation}
[X_1^{(2)},X_2^{(1)}]=[X_1,[X_1,X_2]].
\end{equation}

	 \begin{proposition}[Baker--Campbell--Hausdorff formula, Dynkin's version]

	\begin{equation} \label{eq:BCDH} \log e^{t A_1} e^{tA_2}= \sum_{m=1}^\infty \frac{t^m}{m}  \sum_{j=1}^m \frac{(-1)^{j+1}}{j} \sum_{\substack{u_1+v_1+\dots+u_j+v_j=m\\ u_i+v_i \geq 1}}  \frac{[A_1^{(u_1)}, A_2^{(v_1)},\cdots , A_1^{(u_j)}, A_2^{(v_j)}]}{u_1!v_1 ! \cdots u_j!v_j!}.
	\end{equation}
\end{proposition}
As before, we point out that the series on the right-hand side of \eqref{eq:BCDH} is in general not convergent and has to be interpreted as a formal series. The point here is that the coefficients on the right-hand side of~\eqref{eq:BCDH} match the ones on the right-hand side of~\eqref{eq:preBCH}. The proof of this fact is purely algebraic, see, for example, \cite{BF2,Dynkin}.
\subsection{Proof of Theorem \ref{thm:decompB}}
Now we return to the proof of Theorem \ref{thm:decompB}.

We will need the following two lemmas.
\begin{lemma}\label{lem:B+B-A}  Let $B_\pm$ be as in Theorem \ref{thm:decompB}. Then,
	for
	$ \ell \in \mathbb N$,
	$$B^\ell=(B_-+B_+)^\ell =\sum_{k=0}^\ell \begin{pmatrix} \ell \\ k \end{pmatrix}  B_-^k B_+^{\ell-k} + A,$$
	where $A$ is a matrix such that $AQ_s =0$ for some $s \in \mathbb N$ that depends on $B$ and $\ell$ only.
\end{lemma}
\begin{proof}
Choose any $\sigma \in \{+,-\}^\ell$ and let $\ell_+$ and $\ell_-$ be the number of $+$ and $-$ entries in $\sigma$, respectively.
	By successively changing the order of $B_-$'s and $B_+$'s, it is not hard to see that
	$$\prod_{j =1}^\ell B_{\sigma(j)}=B_-^{\ell_+} B_+^{\ell_-}+A_\sigma,$$
	where $A_\sigma$ is a linear combination of terms with each term being a product of $B_-$'s, $B_+$'s and $[B_-,B_+]$ in a certain order. By assumption we have that $[B_-,B_+]$ has only finitely many non-trivial columns. Clearly, this structure is preserved if we multiply this commutator by an arbitrary matrix from the left. This also holds true if we multiply the commutator from the right by a lower triangular matrix. Finally, if we multiply from the right by an upper triangular matrix that is \emph{banded}, then  the number of non-trivial columns increases with the bandwidth $b$. Still, there are only finitely many non-trivial columns. Then there exists $s_\sigma$ such that $AQ_{s_\sigma} =0$. Hence the statement follows by taking $s=\max_{\sigma \in \{+,-\}^\ell}s_\sigma.$
\end{proof}
\begin{lemma} \label{lem:TrPnB+A} Let $B_\pm$ be as in Theorem \ref{thm:decompB}. Then
	$$\Tr P_n[B_+,A]P_n=\Tr P_n[B_-,A]P_n= 0.$$
	for any matrix $A$ such that  $AQ_s =0$ for some $s \in \mathbb N$ and for $n$ sufficiently large.
\end{lemma}
\begin{proof}
	Let us start with  $\Tr P_n[B_+,A]P_n=0$. Then for $n$ sufficiently large,
	$$\Tr P_n B_+A P_n=\sum_{j=1}^n \sum_{k=1}^\infty (B_+)_{jk} A_{kj}= \sum_{j=1}^s \sum_{k=1}^\infty   (B_+)_{jk} A_{kj}=\sum_{j=1}^{s} \sum_{k=1}^{s+b}   (B_+)_{jk} A_{kj}. $$
	 On the other hand,
	$$\Tr P_n A B_+ P_n=\sum_{k=1}^n \sum_{j=1}^\infty A_{kj}  (B_+)_{jk} = \sum_{k=1}^n \sum_{j=1}^s  A_{kj}  (B_+)_{jk} =\sum_{k=1}^{s+b} \sum_{j=1}^{s}  A_{kj}  (B_+)_{jk}. $$
 Therefore, $ \Tr P_n[B_+,A]P_n=0$ for $n$ sufficiently large.

 The proof of $\Tr P_n[B_-,A]P_n=0$ is even easier. Since $A= AP_n$ for $n\ge s$ and $P_nB_-=P_n B_-P_n$ we have $P_n [B_-, A]P_n= P_n B_-P_n AP_n-P_n A P_n B_-P_n$. Each term is a product of two finite square matrices, and for such products the trace is cyclic.
\end{proof}
\begin{corollary}\label{cor:vanishingofnestedcom} Let $B_\pm$ be as in Theorem \ref{thm:decompB}.
	Let $m\geq 3$, $j \in \mathbb N$ and set $u_1+v_1+\dots+u_j+v_j=m$ where $u_i+v_i \geq 1$ for $i=1,\ldots, j$. Then $$\Tr P_n [B_-^{(u_1)}, B_+^{(v_1)},\cdots , B_-^{(u_j)}, B_+^{(v_j)}]P_n=0,$$
	for $n$ sufficiently large.
\end{corollary}
\begin{proof}
If $v_j>1$, the nested commutator is trivially zero. If $v_j=1$, then
\begin{equation}
[B_-^{(u_1)}, B_+^{(v_1)},\cdots , B_-^{(u_j)}, B_+^{(v_j)}]=[X_1,X_2,\dots,X_r],
\end{equation}
where $r=u_1+v_1+\dots +u_j+v_j$,
 and where each $X_i=B_-$ or $B_+$, and $X_{r-1}=B_-$ and $X_r=B_+$. Let
\begin{equation}
Y_i=[X_i,X_{i+1},\dots,X_r],\qquad i=1,2,\dots,r-1.
\end{equation}
Then $Y_{r-1}=[B_-,B_+]$, and by assumption it has a finite number of non-trivial columns. We assume that this property holds true for $Y_i$, and prove the same for $Y_{i-1}$.
By definition, $Y_{i-1}=[X_{i-1},Y_i]$. Since $Y_i$ has only a finite number of non-trivial columns by assumption, and $X_{i-1}$ is either lower triangular or upper triangular with finite bandwidth, it follows that both $X_{i-1}Y_i$ and $Y_iX_{i-1}$ have only a finite number of non-trivial columns. The same property therefore holds true for  $Y_{i-1}$ and by induction for $Y_2$. 
By Lemma~\ref{lem:TrPnB+A} we obtain ${\rm Tr} P_nY_1P_n=0$, which proves the corollary for $v_j=1$.

If $v_j=0$ and $u_j>1$, then the nested commutator is trivially zero. If $v_j=0$ and $u_j=1$, the arguments above apply, with the right-most part of the nested commutator being equal to $[B_+,B_-]$ instead of $[B_-,B_+]$.

\end{proof}
Now we are ready for the proof of Theorem \ref{thm:decompB}.
\begin{proof}[Proof of  Theorem \ref{thm:decompB}]
	
	Our starting point is the following expression from the proof of Theorem~\ref{lemmacumu}:
	\begin{multline}
	{\rm Tr} P_nB^{l_1}P_n \dots B^{l_j}P_n-{\rm Tr} P_nB^mP_n\\=-\sum_{s=1}^n\sum_{r_1,\dots,r_{j-1}\in I_{j,n}}\left(B^{l_1}\right)_{sr_1}\left(B^{l_2}\right)_{r_1r_2}\dots \left(B^{l_j}\right)_{r_{j-1}s},
	\end{multline}
	where
	\begin{equation}\nonumber
	I_{j,n}=\{r_1,\dots,r_{j-1}: r_i>n \textrm{ for some }i=1,\dots,j-1\}.
	\end{equation}
	As argued in the proof of Theorem~\ref{lemmacumu}, the one-side banded structure shows that this only depends on $s$ and $r_j$ that are within a certain distance of $n$. In particular, one has $s>n-m^2 b$ and $r_i>n-m^2 b$ for all $i=1,2,\dots,j-1$, and thus by Lemma \ref{lem:B+B-A}, it follows that, for sufficiently large $n$,
\begin{multline}
{\rm Tr} P_nB^{l_1}P_n \dots B^{l_j}P_n-{\rm Tr} P_nB^mP_n\\=-\sum_{u_1=0}^{l_1} \cdots 	\sum_{u_j=0}^{l_j} \begin{pmatrix} l_1 \\ u_1\end{pmatrix}\cdots \begin{pmatrix} l_j \\ u_j\end{pmatrix}\\ \Tr \sum_{s=1}^n\sum_{r_1,\dots,r_{j-1}\in I_{j,n}}\left(B_-^{u_1} B_+^{v_1}\right)_{sr_1}\left(B_-^{u_2} B_+^{v_2}\right)_{r_1r_2}\dots \left(B_-^{u_j} B_+^{v_j}\right)_{r_{j-1}s},
\end{multline}
where $v_j= l_j-u_j$. Note that we used that contributions from various matrices $A$ (see lemma  \ref{lem:B+B-A}) disappear since $s>n-m^2 b$ and $r_i>n-m^2 b$.

 Therefore
		\begin{multline}
\Tr \left( \frac{P_nB^{l_1}P_n \dots B^{l_j}P_n- P_nB^mP_n}{l_1!\cdots l_j!}\right)\\=
\Tr \sum_{u_1=0}^{l_1} \cdots 	\sum_{u_j=0}^{l_j} \frac{   P_nB_-^{u_1}  B_+^{v_1}P_n \dots P_n B_-^{u_j} B_+^{v_j} P_n- P_n B_-^{u_1}  B_+^{v_1} \dots  B_-^{u_j} B_+^{v_j}P_n}{u_1!v_1 ! \cdots u_j! v_j!}
 \end{multline}

 By the triangular structure, we have $ P_n B_-=P_n B_- P_n$ and $B_+P_n=P_nB_+P_n$, and hence we can write

 	\begin{multline}
 \Tr \left( \frac{P_nB^{l_1}P_n \dots B^{l_j}P_n- P_nB^mP_n}{l_1!\cdots l_j!}\right)=\\
	  \Tr \sum_{u_1=0}^{l_1} \cdots 	\sum_{u_j=0}^{l_j}  \frac{   (P_nB_- P_n)^{u_1}  (P_n B_+ P_n)^{v_1} \dots (P_n B_- P_n)^{u_j} (P_nB_+P_n)^{v_j} }{u_1!v_1 ! \cdots u_j! v_j!}
	  \\
	  - \Tr 	\sum_{u_1=0}^{l_1} \cdots 	\sum_{u_j=0}^{l_j} \frac{ P_n B_-^{u_1}  B_+^{v_1} \dots  B_-^{u_j} B_+^{v_j}P_n}{u_1!v_1 ! \cdots u_j! v_j!}
	\end{multline}
	for $n$ sufficiently large. Now, by Corollary \ref{cor:trexpa} we can ignore the contribution of the first term to $C_m^{(n)}(B)$ and write
		\begin{multline}
	C_m^{(n)}(B)
	=\\-m! \Tr  \sum_{j=1}^m \frac{(-1)^{j+1}}{j} \sum_{\substack{u_1+v_1+\dots+u_j+v_j=m\\ u_i+v_i \geq 1}} 	\frac{   P_n B_-^{u_1}  B_+^{v_1} \dots  B_-^{u_j} B_+^{v_j}P_n}{u_1!v_1 ! \cdots u_j! v_j!}\\
	=-m! \Tr  P_n \left( \sum_{j=1}^m \frac{(-1)^{j+1}}{j} \sum_{\substack{u_1+v_1+\dots+u_j+v_j=m\\ u_i+v_i \geq 1}} 	\frac{   B_-^{u_1}  B_+^{v_1} \dots  B_-^{u_j} B_+^{v_j}}{u_1!v_1 ! \cdots u_j! v_j!}\right) P_n,
	\end{multline}
	for $m \geq 2$. By \eqref{eq:preBCH} and the Baker--Campbell--Hausdorff theorem we can rewrite this as nested commutators. Then by applying Corollary \ref{cor:vanishingofnestedcom} we see that $C_m^{(n)}(B)=0$ for $m \geq 3$ as claimed.
\end{proof}

	\section{Proof of Theorem \ref{thm:varying}}
	In this section we prove our main result, Theorem \ref{thm:varying}. The starting point is that the moment-generating function for the linear statistic can be written as a determinant involving the exponential of a one-sided banded matrix. We then employ the results from the previous section to prove Theorem \ref{thm:varying}.

\subsection{Cumulants}

 We start by recalling the definition of the cumulants for $X_n(f)= \sum_{j=1}^n f(x_j)$.
For $r=1,2,\dots$, consider
 $$ \mathbb E[\exp_r  (\lambda X_n(f))]=1+\sum_{j=1}^r \frac{\lambda^jM_j^{(n)}(f)}{j!},$$ where $M_j^{(n)}(f)=\mathbb E\left[\left(X_n(f)\right)^j\right]$ are the moments (we chose to work with $\exp_r$ since the moment-generating function is not necessarily well-defined. An alternative is to work with formal expressions).  The cumulants $C_m(X_n(f))$ are defined by
\begin{equation}\label{defCum}
\log \mathbb E\left[\exp_r (\lambda X_n(f))   \right]= \sum_{m=1}^rC_m (X_n(f))\frac{\lambda^m}{m!}+\mathcal O\left(\lambda^{r+1}\right),
\end{equation}
for $\lambda$ in a neighbourhood of $0$.
It is easily verified that $C_m(X_n(f))$ defined in such a manner is indeed independent of $r>m$.

The first step is to give a formula for the cumulants $C_m(X_n(f))$ in terms of $f(J)$. We start with the following observation.

\begin{lemma}
Let $\vec k_i$ be a path satisfying \eqref{eq:paths}, and let $x_j$ for $j=1,\ldots,n$ be points randomly taken from \eqref{eq:MOPE} with $\vec k=\vec k_n$. Assume that $\{\mu_j(x)\}_{j=1}^m$ defined in \eqref{eq:orthomeasures} is a perfect system. Then for any $\lambda\in \mathbb C$ and polynomial $f$, we have
\begin{equation}\label{Eexpr}
\mathbb E\left[\exp_r\left(\lambda \sum_{j=1}^nf(x_j)\right)    \right]=\det (Q_n+ P_n\exp_r(\lambda f(J))P_n)+\mathcal O\left(\lambda^{r+1}\right),
\end{equation}
for $\lambda$ in a neighbourhood of $0$,
with  $J$ as defined in \eqref{eq:generalJacobi}.
\end{lemma}
\begin{remark}
	We will allow $J,\mu$ and $w$ to depend on $n$. To avoid cumbersome notation we will supress the $n$-dependence  and just write  $J,\mu$ and $w$ from now on.
\end{remark}
\begin{proof}
	This lemma can be found in \cite[Lemma 4.1]{Duits2} in a more general setup. For clarity and completeness we include a proof here.
	
	We start with a standard biorthogonalization procedure. Let $G$ be the $n\times n$ matrix given by
	\begin{equation} G_{i,j}=\int g_i (x)p_{\vec k_{j-1}} (x)d\mu(x), \qquad i,j=1,\dots,n.\end{equation} 
	Then, by Andreief's identity, it follows that
	\begin{multline} \label{Andreief}
	\frac{1}{n!}\int\dots \int \det\left(p_{\vec k_{j-1}}(x_i) \right)_{i,j=1}^n \det \left(g_j(x_i)\right)_{i,j=1}^n \prod_{i=1}^n d\mu(x_i)
	\\=\det \left(\int   p_{\vec k_{i-1}}(x)g_{j}(x)d\mu(x)\right)_{i,j=1}^n=\det G,
	\end{multline}
	and thus the normalizing constant for \eqref{eq:MOPE} equals  $Z_{\vec{k}}=n!\det G $. Since \eqref{eq:MOPE} is assumed to be a probability measure, this determinant does not vanish and thus $G$ is invertible. If we further define
	\begin{equation} \begin{pmatrix}
	\tilde  g_1(x)\\
\tilde g_2(x)\\
	\vdots \\
\tilde g_{n}(x)
	\end{pmatrix}
	=G^{-1}
	\begin{pmatrix} 
	g_1(x)\\
g_2(x)\\
	\vdots \\
	g_{n}(x)
	\end{pmatrix},
	\end{equation}
	then 
	\begin{align*}\int \tilde g_{i}(x)p_{\vec k_{j-1}}(x)d\mu(x)&=\sum_{l=1}^n(G^{-1})_{i,l}\int g_{l}(x) p_{\vec k_{j-1}}(x) d\mu(x)\\
	&=\sum_{l=1}^n(G^{-1})_{i,l} G_{l,j}= \delta_{i,j}.
	\end{align*}
	showing that $\{p_{\vec k_{j}}\}_{j=0}^{n-1}$ and $\{\tilde g_j\}_{j=1}^n$ are biorthogonal. From the biorthogonality we find that 
	\begin{equation}\label{eq:jacobimonomial}
	\int  x^\nu p_{\vec k_{i}}(x) \tilde g_{j}(x)d\mu(x)=\int \sum_{s=0}^{i+\nu}\left(J^\nu\right)_{i,s}p_{\vec k_{s}}(x) \tilde g_j(x)d\mu(x)=
	\left(J^{\nu}\right)_{i,j},
	\end{equation}
	for $\nu=1,2,3,\dots$ and $i,j=1,\ldots, n$. Also, by standard row and column operations, the measure \eqref{eq:MOPE} can be written as
	\begin{equation}\label{measure2}
	\frac{1}{n!} \det \big( p_{\vec k_{j-1}}(x_i)\big)_{i,j=1}^{n}\det\big(\tilde g_j (x_i)\big) _{i,j=1}^{n} \prod_{j=1}^{n} {\rm d} \mu(x_j),
	\end{equation}
	where the normalization constant is simply $n!$ by Andrei\'ef's identity and biorthogonality.

Now that we have finished the biorthogonalization procedure, let us return to the claim of the lemma. We start by noting that
	\begin{equation}
	\mathbb E\left[\exp_r\left(\lambda \sum_{j=1}^nf(x_j)\right)    \right]=
	\mathbb E\left[\prod_{j=1}^n \exp_r\left(\lambda f(x_j)\right)    \right]+\mathcal O\left(\lambda^{r+1}\right),
	\end{equation}
	for $\lambda$ in a neighbourhood of $0$. 
	Then it is easily verified, using Andrei\'ef's identity, that 
	$$
\mathbb E\left[\prod_{j=1}^n \exp_r\left(\lambda f(x_j)\right)    \right]=	\det \left(\int  \exp_r(\lambda f(x))p_{\vec k_{i-1}}(x) \tilde g_{j}(x) d\mu(x)\right)_{i,j=1}^{n}.
	$$
	and since $\exp_r(\lambda f(x))$ is a sum of monomials (in $x$), the lemma follows from \eqref{eq:jacobimonomial}.
\end{proof}

Together with Lemma \ref{lem:cumutraces} we thus find the following expression for the cumulants.
\begin{corollary} \label{cor:cumu}
For $m,n=1,2, \ldots$,
	\begin{equation}\label{eqCum}
	C_m(X_n(f))=C_m^{(n)}(f(J)),
	\end{equation}
	where
	$$C_m^{(n)}( f(J))=m!\sum_{j=1}^m \frac{(-1)^{j+1}}{j}\sum_{\substack{l_1+\dots+l_j=m\\ l_i\geq 1}}\frac{{\rm Tr} P_n f(J)^{l_1}P_n\dots  f(J)^{l_j} P_n}{l_1!\dots l_j!},$$
and $P_n$ is the projection onto the first $n$ dimensions of $\ell_2(\mathbb N)$.
	\end{corollary}

\subsection{Comparison with the right limit}

Our next step is to show that the results of Section \ref{sec:one-sided} imply the following lemma.

 \begin{lemma} \label{lem:fJtoTfa}
	Under the conditions of Theorem \ref{thm:varying} we have that
	\begin{equation}
	\lim_{n\to \infty} C_m^{(n)}\left(f(J)\right)-C_m^{(n)}\left(T_{f \circ c} \right)=0,
	\end{equation}
	for any polynomial $f$, with
	\begin{equation} \label{eq:defr}
	\left(f \circ c\right)(z)= f\left(z+\sum_{j=1}^k \frac{a_j}{z-b_j}\right),
	\end{equation}
	and $T_{f \circ c}$ is
    given by the matrix defined in~\eqref{eq:matrixTc} but with $\tau_c$  replaced by $\tau_{f\circ c}$.
\end{lemma}

Before we come to the proof of this lemma, we first mention an easier consequence of Theorem~\ref{lemmacumu}.

 \begin{corollary}\label{cor:fJTa}
 	Under the conditions of Theorem \ref{thm:varying} we have that
	 	\begin{equation}
 			\lim_{n\to \infty} C_m^{(n)}\left(f(J)\right)-C_m^{(n)}\left(f\left(T_c \right)\right)=0,
	 	\end{equation}
	 	for any polynomial $f$.
 \end{corollary}
\begin{proof}
	Since $J$ is banded from one-side and $f$ is a polynomial, Theorem~\ref{thm:rightlimitMOPnevai} implies
\begin{equation}\label{eqnRightLim2}
(f(J))_{n+\ell,n+m}-\left(f\left(T_c\right)\right)_{n+\ell,n+m}\to 0,\end{equation}
as $n\to \infty$. The statement now follows from Theorem \ref{lemmacumu}.
\end{proof}

\begin{lemma}\label{LemmaFin}
	Let $f$ be a polynomial.
	Then, for any polynomial $p$,
	\begin{equation}
	\tau_{f \circ c}p-f\left(\tau_c \right)p
	\end{equation}
	is a polynomial of degree less than or equal to $\deg f-1$.
\end{lemma}
\begin{proof}
	We recall that $M_c$ is the multiplication operator on $\mathcal R$ with multiplier $c(z)=z+ \sum_{j=1}^m \frac{a_j}{z-b_j}$.
		By definition, we have
	\begin{equation}\tau_{f \circ c}-f\left(\tau_c \right)={ P}f(M_c)P^*-f(PM_c P^*) \end{equation}
	and since $f$ is a polynomial, it suffices to show that $$({ P}M_c^jP^*)p-({ P}M_cP^*)^j p,$$ is a polynomial of degree less than or equal to $j-1$, for any polynomial $p$. It is clear that this holds for $j=1$. We assume that it is true for $j-1$ ($j \ge 2$), and prove it for $j$.
We rewrite
	\begin{align*}
	({ P}M_c^jP^*-({ P}M_cP^*)^j)p & =
    PM_c \left(M_c^{j-1} P^* - P^* (PM_c P^*)^{j-1}\right) p
    \\
    & =
    { P}M_c\left(M^{j-1}_c-{P^* P}M_c^{j-1}\right)p-{P}M_cq_{j-2},
	\end{align*}
	where on the last step we use the induction hypothesis to replace $(PM_c P^*)^{j-1} p$ with ${ P}M_c^{j-1}P^* p + q_{j-2}$, where $q_{j-2}$ is some polynomial of degree less than or equal to $j-2$. The first term ${ P}M_c\left(M^{j-1}_c-{P^*P}M_c^{j-1}\right)p$ is a constant (indeed, for any rational function $g$ we have that $P M_c(g-P^* Pg)$ is a constant) and ${P}M_c q_{j-2}$ has degree less than or equal to $j-1$, and thus we have proven the lemma.
\end{proof}

	 \begin{proof}[{Proof of Lemma \ref{lem:fJtoTfa}}]
	 	By Lemma \ref{LemmaFin} we see that the matrices $T_{f \circ c} $ and $f(T_c)$ only differ possibly in the first $\deg f-1$ columns.
Hence, as $n\to \infty$,
	 	\begin{equation}(T_{f \circ c})_{n+\ell, n+m}-f(T_c)_{n+\ell,n+m}\to 0,\end{equation}
	 	for any fixed $\ell,m$. Thus, the statement follows from Theorem \ref{lemmacumu} and~\eqref{eqnRightLim2}.
	 \end{proof}

\subsection{Computing the cumulants for the right-limit}
We now compute the limiting behavior of $C_m^{(n)}\left(T_{f \circ c} \right)$ as $n \to \infty$.  For the rest of the section let us put $$r = f \circ c$$ for some given polynomial $f$.

We start with the following lemma.

\begin{lemma}\label{lem:decompr} 
Split $r=r_++r_-$ where $r_+$ is the polynomial part of $r$, i.e. $r_+=Pr$. Then
	\begin{enumerate}
			\item $T_r=T_{r_+}+T_{r_-}$.
			\item $T_{r_+}$ is upper triangular and banded (more precisely, $(T_{r_+})_{j,j+k}=0$ for $k<0$ or $k > \deg r_+$),
			\item $T_{r_-}$ is strictly lower triangular,
			\item $[T_{r_-},T_{r_+}]Q_s=0$ 
for $s\geq \deg r_+-1$.
	\end{enumerate}
\end{lemma}
\begin{proof}
	1. This is trivial.
	
	2. By construction of the basis it follows that $r_+(z) \pi_{\vec k_{n}}$ can be expressed in term of $\pi_{\vec k_m}$ with only $n \leq m \leq n+ \deg r_+$. Therefore, $T_{r_+}$ is upper  triangular and  banded.
	
	3. The polynomial part of $r_-p$ for any polynomial $p$ is a polynomial of strictly lower degree.

	4.  Note that 	$$
	\left(\tau_{r_-} \tau_{r_+}-\tau_{r_+} \tau_{r_-}\right)p(z)= \frac{1}{2 \pi i } \oint_\gamma \frac{r_+(w)-r_+(z)}{w-z} r_-(w) p(w) dw,
$$
	where $\gamma$ is a simple counter-clockwise oriented contour that goes around all poles of $r_- p$.  Now, $ \frac{r_+(w)-r_+(z)}{w-z} $ is a polynomial of degree $\deg r_+-1$ in $z$. After going to the  basis $\{ \pi_{\vec k_n}\}_n$ we see therefore that only the first $\deg r_+-1$ columns of $[T_{r_-},T_{r_+}]$  contain non-zero entries. This proves the statement with $s \geq \deg r_+-1$.  \end{proof}

By Theorem  \ref{thm:decompB} and  Lemma \ref{lem:decompr} we readily find the following corollary.

\begin{corollary}\label{cor:cumutozero}
	 Let $f$ be a polynomial. Then
	$$\lim_{n \to \infty} C_m^{(n)} (T_{r})=0, $$
	for $m\geq 3$.
	\end{corollary}

It remains to compute the limiting behavior of $C_2^{(n)}(T_r)$. For this, we will change basis and work with $C_2^{(n)}(\mathcal{T}_r)$ instead. This change of basis is surprisingly delicate.
\begin{lemma} \label{lem:variance0} Let $f$ be a polynomial. Then
	$$C_2^{(n)}(T_{r})=C_2^{(n)}(\mathcal{T}_{r}),$$
	for $n$ large enough.
	\end{lemma}
\begin{proof} We will use the 
decomposition $r=r_++r_-$ with $r_+=Pr$.
	We wish to rely on the connection between $T_r$ and $\mathcal T_r$ from Lemma \ref{S-1TS}.	
First observe that

$$  C_2^{(n)}(T_r)=- \Tr P_n T_r P_n T_r P_n+\Tr P_n (T_r^2) P_n= \Tr P_n T_r Q_n T_r P_n.$$
Because of the triangular structure,
\begin{equation}
P_nT_{r_-}Q_n=0, \qquad Q_n T_{r_+}P_n=0,
\end{equation}
and thus we can write
$$ C_2^{(n)}(T_r)= \Tr P_n T_{r_+} Q_n T_{r_-} P_n.$$
Now, recalling Lemma \ref{S-1TS}, we change basis from $\{\pi_{\vec k_n}\}_n$ to the canonical basis $\{z^n\}_n$. This change of basis shows that  $T_{r_\pm}$ and $\mathcal T_{r_{\pm}}$ are related by conjugation with  a lower triangular matrix $S$,
$$T_{r_{\pm}}=S \mathcal T_{r_{\pm}} S^{-1}.$$
Thus, we can write
\begin{equation} \label{eq:Variancestep1}
 C_2^{(n)}(T_r)= \Tr P_n S \mathcal T_{r_+} S^{-1} Q_n S \mathcal T_{r_-}S^{-1} P_n.
\end{equation}
It remains to argue that we can drop the $S^{\pm 1}$ everywhere on the right-hand side. This is is in fact not obvious since $P_n$ and $S^{\pm 1}$ do not commute.

By the lower-triangular structure of $S$, we have
$$ P_nS^{\pm 1} = P_n S^{\pm1} P_n \quad \text{and} \quad  S^{\pm 1}Q_n= Q_n S^{\pm 1} Q_n.$$
This implies that
\begin{equation} \label{eq:PSPSinP}
	P_n S P_n S^{-1}P_n=P_n \quad  \text {and} \quad Q_n S Q_n S^{-1} Q_n= Q_n.
\end{equation}
This allows us to rewrite \eqref{eq:Variancestep1} to
$$  C_2^{(n)}(T_r)= \Tr P_n SP_n \mathcal T_{r_+} Q_nS^{-1} Q_n S  \mathcal T_{r_-}  S^{-1} P_n.$$
 Now insert $I=P_n+Q_n$ after the $\mathcal T_{r_-}$ at the right-hand side to obtain
\begin{multline}
C_2^{(n)}(T_r)= \Tr P_n SP_n \mathcal T_{r_+} Q_nS^{-1} Q_n S  \mathcal T_{r_-} P_n S^{-1} P_n\\
+\Tr P_n SP_n \mathcal T_{r_+} Q_nS^{-1} Q_n S  \mathcal T_{r_-} Q_n S^{-1} P_n.
\end{multline}
Since $\Tr A B= \Tr B A$ if $A,B$ are of finite rank and by \eqref{eq:PSPSinP}, we can simplify the first term at the right-hand side to
\begin{multline}
C_2^{(n)}(T_r)=  \Tr P_n \mathcal T_{r_+} Q_nS^{-1} Q_n S  \mathcal T_{r_-} P_n \\
+\Tr P_n SP_n \mathcal T_{r_+} Q_nS^{-1} Q_n S  \mathcal T_{r_-} Q_n S^{-1} P_n.
\end{multline}
By inserting $I=P_n+Q_n$ before $\mathcal T_{r_-}$ in the first term at the right-hand side we find
\begin{multline}
C_2^{(n)}(T_r)=  \Tr P_n \mathcal T_{r_+} Q_nS^{-1} Q_n S  Q_n\mathcal T_{r_-} P_n \\
+\Tr P_n \mathcal T_{r_+} Q_nS^{-1} Q_n S  P_n\mathcal T_{r_-} P_n \\
+\Tr P_n SP_n \mathcal T_{r_+} Q_nS^{-1} Q_n S  \mathcal T_{r_-} Q_n S^{-1} P_n.
\end{multline}
In the first term at the right-hand side we use \eqref{eq:PSPSinP}. For the third term, we note that since $\mathcal T_{r_-}$ is lower triangular  we have $ \mathcal T_{r_-}Q_n= Q_n \mathcal T_{r_-}Q_n$ and combining this with \eqref{eq:PSPSinP} gives
\begin{multline}
C_2^{(n)}(T_r)=  \Tr P_n \mathcal T_{r_+}   Q_n\mathcal T_{r_-} P_n \\
+\Tr P_n \mathcal T_{r_+} Q_nS^{-1} Q_n S  P_n\mathcal T_{r_-} P_n \\
+\Tr P_n SP_n \mathcal T_{r_+} Q_n \mathcal T_{r_-} Q_n S^{-1} P_n.
\end{multline}
Using the cyclicity of the trace in the second and third term we find
\begin{multline}
C_2^{(n)}(T_r)=  \Tr P_n \mathcal T_{r_+}   Q_n\mathcal T_{r_-} P_n \\
+\Tr P_n\mathcal T_{r_-}P_n \mathcal T_{r_+} Q_nS^{-1} Q_n S  P_n \\
+\Tr P_n \mathcal T_{r_+} Q_n \mathcal T_{r_-} Q_n S^{-1} P_nSP_n.
\end{multline}
By the triangular structure we have $P_n \mathcal T_{r_-}P_n=P_n\mathcal T_{r_-}$ and $Q_n\mathcal T_{r_-}Q_n=\mathcal T_{r_-}Q_n$, so we can drop the $Q_n$ and $P_n$ between the matrices $\mathcal T_{r_\pm}$ and get
\begin{multline}
C_2^{(n)}(T_r)=  \Tr P_n \mathcal T_{r_+}   Q_n\mathcal T_{r_-} P_n \\
+\Tr P_n\mathcal T_{r_-} \mathcal T_{r_+} Q_nS^{-1} Q_n S  P_n \\
+\Tr P_n \mathcal T_{r_+}  \mathcal T_{r_-} Q_n S^{-1} P_nSP_n.
\end{multline}
Since $Q_nP_n =O$ we have $Q_n S^{-1} Q_n S P_n=-Q_nS^{-1}  P_n S P_n$ and thus
$$
C_2^{(n)}(T_r)=  \Tr P_n \mathcal T_{r_+}   Q_n\mathcal T_{r_-} P_n
+\Tr P_n [\mathcal T_{r_+},  \mathcal T_{r_-}] Q_n S^{-1} P_nSP_n.
$$
Now, by the same argumentation as in Lemma \ref{lem:decompr}, we have $  [\mathcal T_{r_+},\mathcal T_{r_-}]  Q_n=O$ for $n$ large enough, and thus,
\begin{equation}\label{eq:variance}
C^{(n)}_2(T_r)= \Tr P_n\mathcal T_{r_+}Q_n  \mathcal T_{r_-} P_n= C_2^{(n)}(\mathcal T_r)
\end{equation}
which proves the statement.
\end{proof}
The benefit of working with $C_2^{(n)}( \mathcal T_r)$ over $C_2^{(n)}(T_r)$ is that it is much easier to compute.
\begin{lemma}\label{lem:variance} Let $f$ be a polynomial. 
Then
	$$\lim_{n\to \infty}C_2^{(n)}(\mathcal T_r)= \sum_{j=1}^\infty j r_j r_{-j}.$$
\end{lemma}
\begin{proof}
	This follows directly by  rewriting \eqref{eq:variance}
	$$\Tr P_n\mathcal T_{r_+}Q_n  \mathcal T_{r_-} P_n= \sum_{k=1}^n \sum_{\ell=k}^\infty r_{\ell} r_{-\ell}= \sum_{j=1}^{\infty} \min(j,n) r_jr_{-j},$$
 and then taking the limit	 as $n \to \infty$.
\end{proof}

	\subsection{Proof of Theorem  \ref{thm:varying}}
	\begin{proof}
		To prove that $X_n(f)-\mathbb E X_n(f)$  converge to a Gaussian in distribution, it is sufficient (1) to show that $C_m(X_n(f)) \to 0$ as $n \to \infty$ for $m \geq 3$ and (2) to compute $C_2(X_n(f))$.  Corollary \ref{cor:cumu} shows that these cumulants are given by $C_m^{(n)}(f(J))$. By Lemma \ref{lem:fJtoTfa} it suffices to consider the limiting behavior of $C_m^{(n)}(T_{f \circ c}))$ as $n\to \infty$. Corollary \ref{cor:cumutozero} tells us that the higher cumulants indeed tend to zero as $n \to \infty$. The second cumulant is computed using Lemmas \ref{lem:variance0} and \ref{lem:variance}.
 \end{proof}

	\section{Applications} \label{sec:examples}
		In this section we will illustrate the main results with some applications. We will prove CLT's for the MOPE's related to the multiple Hermite, Laguerre, Charlier, Krawtchouk, and Meixner polynomials and also discuss them in the context of random matrix theory and integrable probability.
		\subsection{GUE with external source and Multiple Hermite polynomials}
		Let us give more details here on the example from the Introduction. So consider \eqref{eq:externalsource} where $H$ is a diagonal matrix with $m$ distinct values $h_1, \ldots, h_m$ on the diagonal with multiplicities $k_1, \ldots, k_m$. Then the eigenvalues of a matrix $M$ chosen randomly from \eqref{eq:externalsource} form a MOPE with weights $$w_\ell(x)= e^{- n\big(\tfrac12 x^2-h_\ell x\big)}, \qquad \ell=1,\ldots,m.$$ The recurrence coefficients of the multiple Hermite polynomials of type II, denoted by $p_{\vec k}$ for $\vec k=( k_1,\ldots,k_m)$ are then given by \eqref{eq:multiplehermiterec}. We readily verify that the conditions of Theorem \ref{thm:varying} are satisfied which gives us the following result.
\begin{corollary}
    Let $\{\vec k_n\}_{n=0}^\infty$ be a path satisfying~\eqref{eq:paths}. Let $\{x_j\}_{j=1}^n$ be the eigenvalues of a random matrix from GUE with external source as defined in Section~\ref{ss:MOPE}, where the  eigenvalues $h_1, \dots, h_m$ of $H$ have multiplicities  $\vec k_n$ ($h_1,\dots,h_m$ are fixed and independent of $n$ while the sum of their multiplicities varies and is equal to $n$).  Then for any polynomial $f$ with real coefficients, we have
	$$\sum_{j=1}^n f(x_j)- \mathbb E\left[ \sum_{j=1}^n f(x_j) \right] \to N\left(0,\sum_{\ell=1}^\infty \ell f_\ell f_{-\ell}\right),$$
	in distribution as $n\to \infty$, where
	$$f_\ell= \frac{1}{2 \pi i} \oint_\gamma f\left(z+ \sum_{j=1}^m\frac{\nu_j}{z-h_j}  \right) \frac{dz}{z^{\ell +1}}$$
	and $\gamma$ is a contour around the poles $h_j$  with  counter-clockwise orientation, where $\nu_j$ is the limit in \eqref{eq:paths}.
\end{corollary}
\begin{proof}
We have nearest neighbour recurrence relations~\eqref{eq:multiplehermiterec} which has the form of~\eqref{eq:nearest_neighbor} with $a^{(n)}_{\vec k,j} = \frac{k_j}{n}$ and $b^{(n)}_{\vec k,j} = h_j$. We readily verify condition \eqref{eq:multipleNevaivarying}. Namely, for any path $\vec k_n$ satisfying ~\eqref{eq:paths} and any fixed $s\in \mathbb Z$, we have  $a^{(n)}_{\vec k_{n+s},j} \to \nu_j$ and $b^{(n)}_{\vec k_{n+s},j} \to h_j$ as $n  \to \infty$. Then Theorem~\ref{thm:varying} proves our statement.
\end{proof}

To the best of our knowledge, this is the first CLT for this model.  Although one may argue that the point process is determinantal and that there are explicit double integral formulas for the kernel \cite{BK}, which means that proving a CLT using steepest descent techniques on these integrals should be possible, we emphasize that the many technical details in that approach increase with $m$. Our approach makes a very direct verification  possible, without all the cumbersome technical details.
		
		We end this example with a comment on the fact that the results of \cite{BD} are not sufficient for proving the above CLT. As mentioned before, the main assumption for the CLT in \cite{BD} was that the right limit of the  recurrence matrix $J$ in \eqref{eq:generalJacobi} is constant along the  diagonals. We claim that this is only possible for the multiple Hermite polynomials in case $m=1$.  This is most easily understood by an example: Consider the external source model  with $m=2$,  the values $h_1\neq h_2$ on the diagonal  and multiplicites $(n/2,n/2)$ with $n$ even. Then take the family of polynomials $\{p_{\vec k_j}\}_{j=0}^n$ where $\vec k_j=( \lfloor (j+1)/2\rfloor,\lfloor j/2 \rfloor)$ and $\lfloor x\rfloor$ stands for the largest integer less than or equal $x$.  From the recurrences \eqref{eq:multiplehermiterec} and \eqref{eq:consistency}, we then see that
		 $$x p_{\vec k_j}(x)= \begin{cases}
		p_{\vec k_{j+1}}(x)+  h_1 p_{\vec k_j}(x)+ \frac{ k_{1,j}+k_{2,j}}{n}p_{\vec k_{j-1} }(x)-\frac{ k_{1,j}(h_2-h_1)}{n} p_{\vec k_{j-2} }(x), & j \textrm{ even}\\
			p_{\vec k_{j+1}}(x)+  h_2 p_{\vec k_j}(x)+\frac{ k_{1,j}+k_{2,j}}{n}p_{\vec k_{j-1} }(x)+ \frac{ k_{2,j}(h_2-h_1)}{n} p_{\vec k_{j-2} }(x), & j \textrm{ odd}.
		 \end{cases}$$
		Therefore there is a 2-periodic structure along the diagonals in the recurrence matrix $J$, and every right limit  will have this 2-periodicity.  In fact, the right limit is  a particular block Toeplitz matrix. The proof of the CLT in \cite{BD} only apply to cases where the right limit is a scalar Toeplitz matrix and the extension to block Toeplitz is in general false. For instance, in the multi-cut case for unitary ensembles the right-limit can be a block Toeplitz matrix, but we know that there is no CLT  in the multi-cut case.  See, for example, \cite{BD} for a discussion. The point is that the block Toeplitz matrix in this example with multiple Hermite polynomials is of a very special type.
		
		\subsection{Wishart ensembles and Multiple Laguerre polynomials}
In the next example, we consider the measure
	\begin{equation}
			(\det M)^{\alpha} e^{-n \Tr M \Sigma} \ d M,
	\end{equation}
	on the space of positive definite matrices, where $\alpha>0$ and $\Sigma$ is a diagonal matrix with strictly positive entries. We will study the  case where $\Sigma$ has  precisely $m$ different values $\sigma_1, \ldots, \sigma_m$ on the diagonal with multiplicities $k_1, \ldots, k_m$. It is known that the eigenvalues of $M$ form a multiple orthogonal polynomial ensemble, with  weights now given by
$$
	w_j(x)= x^\alpha e^{-n \sigma_j x},\qquad j=1,\ldots, m.	
$$
The multiple orthogonal polynomials are called multiple Laguerre polynomials \emph{of  the second kind} \cite{V3} and are denoted by $L^{\alpha,\vec \sigma}_{\vec k}$ where $\vec \sigma=(\sigma_1, \ldots, \sigma_m)$. The nearest neighbor recurrences now are
given by
$$
	x L^{\alpha, \vec \sigma}_{\vec k}(x)= L^{\alpha, \vec \sigma}_{\vec k+ \vec e_\ell}(x)+ b_{\vec k, \ell} L^{\alpha, \vec \sigma}_{\vec k}(x)+ \sum_{j=1}^m a_{\vec k,j}L^{\alpha, \vec \sigma}_{\vec k- \vec e_j}(x)
$$
and \cite[§ 3.6.2]{V3}
$$
a_{\vec k,j}= \frac{ k_j(|\vec k|+\alpha)}{n^2 \sigma_j^2}, \qquad   	b_{\vec k, j}= \frac{|\vec k|+ \alpha+1}{n \sigma_j}+ \sum_{r=1}^m \frac{k_r}{n\sigma_r}.
$$
We see directly that also here the conditions of Theorem \ref{thm:varying} are satisfied, so we obtain the CLT of Theorem~\ref{thm:varying} with $a_j = \frac{\nu_j}{\sigma_j^2}$ and $b_j = \frac{1}{\sigma_j} + \sum_{r=1}^n \frac{\nu_r}{\sigma_r}$. Note that we can even let $\alpha$ depend linearly on $n$ which will change the parameters in the CLT in an obvious way. To the best of our knowledge, the CLT for Wishart ensembles has not appeared in the literature before.

\subsection{Discrete multiple orthogonal polynomials}

The examples above are well-known in the literature. It is lesser known that  discrete multiple orthogonal polynomials also appear in integrable probability. We will discuss three families of examples related to multiple Charlier, multiple Krawtchouk, and multiple Meixner polynomials, based on Markov processes for non-colliding particles. It is interesting to note that these Markov processes are special cases of the more general Schur process and  the Multiple Orthogonal Polynomial Ensembles are particular specializations of the Schur measure. This connection is not needed to understand the examples, but we will include a detailed explanation of these claims in the last paragraph for completeness.  We refer  to \cite{BG2,JL} for excellent lecture notes containing more background on the general constructions of this paragraph, including a more detailed discussion about Schur processes and how they appear in integrable probability.

\subsubsection{Markov processes on Weyl chambers}

Let $P_t(x-\xi)$ be the transition kernel (for the probability to jump from $\xi$ to $x$ after time $t$) for a single particle Markov process (time may be discrete or continuous) on $\mathbb Z$. We will mainly be interested in the cases of the Poisson process
\begin{equation}\label{eq:Poisson}
	P_t(x)= \begin{cases}e^{- t } \frac{t^x}{x!}, &  x = 0,1,2,3, \ldots,\\
	0,  & x <0,
	\end{cases} \quad  t >0,
\end{equation}
the simple random walk with binomial transition function
\begin{equation}\label{eq:simpleRandomWalk}
	P_t(x)= \begin{cases}
p^x(1-p)^{t-x} \begin{pmatrix}  t \\ x \end{pmatrix}, & x = 0,1, \ldots, t,\\
	0, & x<0 \textrm{ or } x>t,
	\end{cases}   \quad  t=1,2, \ldots,
\end{equation}
and the random walk where each jump is geometrically distributed with parameter $a\in (0,1)$ (and the parameter is constant in time), having the negative binomial distribution as transition function
\begin{equation}\label{eq:geometricRandomWalk}
P_t(x)= \begin{cases}
(1-a)^t \begin{pmatrix}
t-1+x\\
t-1
\end{pmatrix}a^x, & x = 0,1, \ldots, \\
0, & x<0,
\end{cases}   \quad  t=1,2, \ldots.
\end{equation}
There is a standard  construction for defining non-colliding processes with $n$ particles, starting from these Markov processes, using Doob's $h$-transform \cite{konig}. The key to this construction is the fact that $k \mapsto \alpha^k$ is a harmonic function for the transiton kernel $P_t(x-\xi)$. That is,
$$\sum_x P_t(x-\xi) \alpha^x= c_{t,\alpha} \alpha^\xi,$$
for some constant $c_{t,\alpha}$. On the Weyl Chamber
$$
	W_n= \{x_1 < \ldots < x_n \mid x_j \in \mathbb Z\},
$$
we can then, using the Cauchy--Binet identity, define a Markov process by taking the transition function
\begin{equation} \label{jpdfexample3}
	\frac{1}{n!\prod_i c_{t,\alpha_i}} \det \left( P_t(x_i-\xi_j)\right)_{i,j=1}^n\frac{ \det\left( \alpha_j^{x_i}\right)_{i,j=1}^n }{\det\left( \alpha_j^{\xi_i}\right)_{i,j=1}^n}
\end{equation}
 For now the coefficients $\alpha_j$ are arbitrary, but all with distinct values, in $(0,1]$. Since the states space is $n$-dimensional and ordered, we can think of the full process as a collection of $n$ processes. Each of these processes is driven by $P_t$ but they can never leave the Weyl chamber. We say that the processes are \emph{non-colliding}.
\subsubsection{Choosing special parameters}
We will start Markov process with transition functions \eqref{jpdfexample3} from consecutive integers $\xi_j=j-1$ for $j=1, \ldots,n$. This special initial condition is important for the connection to multiple orthogonal polynomials. Indeed, with this choice the determinant $\det P_t(x_i-\xi_j)$ reduces to the following product  containing the Vandermonde determinant

  $$\det \left(P_t(x_i-j+1)\right)_{i,j=1}^n =  \det(x_i^{j-1})_{i,j=1}^n \prod_{i=1}^n \mu_t(x_j).$$
  The  weight function $\mu_t$ depends on the choice of the transition function  \eqref{eq:Poisson}, \eqref{eq:simpleRandomWalk}, or \eqref{eq:geometricRandomWalk}, and we will do this step case by case in the discussion below.

Apart from  fixing the initial positions we  also specify the values of $\alpha_j$'s. While \eqref{jpdfexample3} was defined with distinct $\alpha_j$'s, we would like to consider the case where some or all are allowed to be equal, and for this purpose we give the following lemma.

\begin{lemma} \label{lem:schurtoMOP} Consider 
\begin{equation}\label{funalpha}\frac{\det\left( \alpha_j^{x_i}\right)_{i,j=1}^n}{ \det\left( \alpha_j^{\xi_i}\right)_{i,j=1}^n}\end{equation}
as a function of $\left(\alpha_j\right)_{j=1}^n$, with  distinct $\alpha_j$'s different from zero. The function \eqref{funalpha} extends to a continuous function for $\left(\alpha_j\right)_{j=1}^n\in (\mathbb R\setminus \{0\})^n$. When several values of $\alpha_j$ are the same, denote $\{\alpha_j\}_{j=1}^n=\{\gamma_\ell\}_{\ell=1}^m$  with  $\gamma_\ell$'s, and let $k_\ell=\#\{j:\alpha_j=\gamma_\ell \textrm{ for }j=1,\dots,n\}$. Then
\begin{equation}\label{limfunalpha} \frac{\det\left( \alpha_j^{x_i}\right)_{i,j=1}^n}{ \det\left( \alpha_j^{\xi_i}\right)_{i,j=1}^n}=\frac{ \det \left( g_j(x_i)\right)_{i,j=1}^n}{ \det\left( g_j(\xi_i)\right)_{i,j=1}^n}, \end{equation}
where the $g_j$'s are given by \eqref{eq:MOPEgj} with $w_j(x)=\gamma_j^x$. 
\end{lemma}
\begin{proof}
	The proof that we will give here is based on a standard procedure using l'H\^opital's rule. The point here is that determinants in \eqref{funalpha} are both zero when $\alpha_j=\alpha_k$ for some $j \neq k$. We therefore have to consider limits  $\alpha_j \to \alpha_k$. 
	
	First note that the ratio in \eqref{funalpha} is a symmetric function of $\alpha_j$'s (in fact, by \eqref{eq:Schurpol} it is a ratio of Schur polynomials and therefore a well-defined continuous function of $\left(\alpha_j\right)_{j=1}^n\in (\mathbb R\setminus \{0\})^n$ for distinct $x_j$ and distinct $\xi_j$). If all $\alpha_j$'s are distinct we may therefore assume that they are ordered by growth $\alpha_1 < \ldots< \alpha_n$. Also assume that $\gamma_\ell$'s are ordered by growth $\gamma_1<\ldots<\gamma_m$. 
 	
 	Now set $\alpha_1=\gamma_1$ and consider the limit $\alpha_2 \to \alpha_1$. As function of $\alpha_2$ the numerator and denominator have a simple zero at $\alpha_1$. Thus, by l'H\^opital's rule we find
 	$$
 	\lim_{\alpha_2 \to \alpha_1} \frac{\det\left( \alpha_j^{x_i}\right)_{i,j=1}^n}{ \det\left( \alpha_j^{\xi_i}\right)_{i,j=1}^n}=\lim_{\alpha_2 \to \alpha_1}\frac{\frac{\partial}{\partial \alpha_2 }\det\left( \alpha_j^{x_i}\right)_{i,j=1}^n}{ \frac{\partial}{\partial \alpha_2 } \det\left( \alpha_j^{\xi_i}\right)_{i,j=1}^n}
 	=\lim_{\alpha_2 \to \alpha_1}\frac{\alpha_2\frac{\partial}{\partial \alpha_2 }\det\left( \alpha_j^{x_i}\right)_{i,j=1}^n}{\alpha_2 \frac{\partial}{\partial \alpha_2 } \det\left( \alpha_j^{\xi_i}\right)_{i,j=1}^n}
 	$$
 	Now the numerator and denominator can computed using the rule
 	$$
 \lim_{\alpha_2 \to \alpha_1}\left(\alpha_2\frac{\partial}{\partial \alpha_2 }\right)\det\left( \alpha_j^{x_i}\right)_{i,j=1}^n=\det
 	\begin{pmatrix}
 	\vdots  & \vdots&\vdots&\cdots \\
 	\alpha_1^{x_i}&	x_i\alpha_1^{x_i}& \alpha_3^{x_i}& \cdots \\
 	\vdots  & \vdots & \vdots&\cdots 
 	\end{pmatrix}
 	$$
 	When we next take $\alpha_3\to \alpha_1$ then it is important to note that  $\alpha _1$ is a double zero of the determinant as a function of $\alpha_3$. We thus have to apply l'H\^opital's rule with the second derivative:  
 	$$\lim_{\alpha_3 \to \alpha_1} \lim_{\alpha_2 \to \alpha_1} \frac{\det\left( \alpha_j^{x_i}\right)_{i,j=1}^n}{ \det\left( \alpha_j^{\xi_i}\right)_{i,j=1}^n}=	\lim_{\alpha_3 \to \alpha_1} \lim_{\alpha_2 \to \alpha_1} \frac{ \left(\alpha_3 \frac{\partial}{\partial \alpha_3 } \right)^2\alpha_2\frac{\partial}{\partial \alpha_2 }\det\left( \alpha_j^{x_i}\right)_{i,j=1}^n}{\left(\alpha_3 \frac{\partial}{\partial \alpha_3 } \right)^2\alpha_2 \frac{\partial}{\partial \alpha_2 } \det\left( \alpha_j^{\xi_i}\right)_{i,j=1}^n}.
 	$$
 	And now 
 	$$¨
 	\lim_{\alpha_3 \to \alpha_1} \lim_{\alpha_2 \to \alpha_1}
 	\left(\alpha_3 \frac{\partial}{\partial \alpha_3 } \right)^2\left(\alpha_2\frac{\partial}{\partial \alpha_2 }\right)\det\left( \alpha_j^{x_i}\right)_{i,j=1}^n=\det \begin{pmatrix}
 	\vdots  & \vdots&\vdots  & \vdots&\cdots \\
 	\alpha_1^{x_i}&	x_j\alpha_1^{x_i}& x_j^2 \alpha_1^{x_i}& \alpha_4^{x_i}&\cdots \\
 	\vdots  & \vdots &\vdots  & \vdots& \cdots 
 	\end{pmatrix}.
 	$$
 	By iterating this procedure the statement follows. For $\alpha_j$ for $j=1, \ldots, k_1$ we apply l'H\^{o}pital's rule with using higher derivatives $(\alpha_j \partial/\partial \alpha_j)^{j-1}$  and this changes the $j$-th column into  $x_i^{j-1}\alpha_1^{x_i}$. Then we start over by setting $\alpha_{k_1+1}=\gamma_2$ and continue applying l'H\^{o}pital's rule for $\alpha_{k_2+2} \to \alpha_{k_1+1}$ and so on. 
 	\end{proof}

By the lemma, we thus see that  in  all three cases of \eqref{eq:Poisson}, \eqref{eq:simpleRandomWalk}, and \eqref{eq:geometricRandomWalk} the positions $x_j$  at time $t$ have the joint probability function  proportional to
\begin{equation} \label{jpdfexample300}
\det(x_i^{j-1})_{i,j=1}^n  \det \left( g_j(x_i)\right)_{i,j=1}^n\prod_{i=1}^n \mu_t(x_j), 
\end{equation}
where the $g_j$'s are given by \eqref{eq:MOPEgj} with $w_j(x)=\gamma_j^x$ and corresponding multiplicity $k_j$,  and thus   form  multiple orthogonal polynomial ensembles.

Our purpose is now to show that the CLT in Theorem \ref{thm:varying} applies when both $t \to \infty$ and $n \to \infty$. In other words, Theorem  \ref{thm:varying} can be used to provide information on  the long time behavior for large collections of non-colliding processes.
\subsubsection{Multiple Charlier Ensemble}
Consider $P_t(u)$ as in \eqref{eq:Poisson} and let us start the process at consecutive integers $\xi_j=j-1$, for $j=1, \ldots,n$. The heart of the matter is the fact that
$$
	P_t(x+1-j)=(-1)^{j+1} t^{-j+1} e^{- t} \frac{t^x}{x!} (-x)_{j-1}, \qquad x \geq 0,
$$
where $(a)_j$ stands for the Pochhamer symbol  $(a)_j=a(a+1)(a+2) \cdots (a+j-1)$. Thus, up to the factor $t^x/x!$, we see that $P_t(x+1-j)$ is a polynomial of degree $j-1$ in $x$. Therefore (after using some standard rules for determinants), we see that
	$$
		\det\left( P_t(x_i-j+1)\right)_{i,j=1}^n = c_n(t) \prod_{i=1}^n \frac{t^{x_i}}{x_i!} \det\left(x_i^{j-1}\right)_{i,j=1}^n,
	$$
for some  $c_n(t)$ independent of the $x_i$'s.
	
Concluding, we find that the joint density \eqref{jpdfexample300} at time $t$  indeed defines a MOPE on the non-negative integers, with $\vec k= (k_1, \ldots, k_m)$, where $|\vec k|=n$, and
	$$
		w_j(x)= \gamma_j^x, \quad j=1, \ldots,m, \qquad  \mu_t(x)= \frac{t^x}{x!}.
	$$
The multiple orthogonal polynomials are called the Multiple Charlier polynomials, denoted by $C_{\vec k}(x)$.
Observe that in the case $m=1$, these are indeed the classical Charlier polynomials.	The nearest neighbor recurrence coefficients for $C_{\vec k}$ have been computed explicitly in \cite{VanAssche11}. By taking the parameter $a_j$ in \cite{VanAssche11}  to be $a_j=t \gamma_j$ we find
$$
	x C_{\vec k}(x)=  C_{\vec k+ \vec e_\ell}(x)+ b_{\vec k, \ell} C_{\vec k}(x)+ \sum_{j=1}^ma_{\vec k,j}C_{\vec k-\vec e_j}(x).
$$		
where
$$
	a_{\vec k,j}= k_j t \gamma_j,  \qquad   	b_{\vec k, \ell}=   t  \gamma_\ell + |\vec k|.
	$$
To get a reasonable limit theorem, we need to rescale the time parameter and space parameters. The explanation of this is that after long time the particles will spread out over a large interval and the CLT applies after an appropriate rescaling of time and space. It turns out the correct rescaling is $t= n \tau$ and $x= n\xi $. By setting $\tilde C_{\vec k}(\xi)=n^{-|\vec k|} C_{\vec k}(n \xi)$ we obtain
$$
	\xi \tilde  C_{\vec k}(\xi)= \tilde  C_{\vec k+ \vec e_\ell}(\xi)+ \tilde b_{\vec k, \ell} \tilde C_{\vec k}(\xi)+ \sum_{j=1}^m\tilde a_{\vec k,j}\tilde C_{\vec k-\vec e_j}(\xi).
$$		
where
$$
\tilde a_{\vec k,j}=k_j  \tau \gamma_j/n  \qquad   	\tilde b_{\vec k, \ell}=  \tau \gamma_\ell + \frac{ |\vec k|}{n}.
$$
Now one readily verifies that the conditions of Theorem \ref{thm:varying} are satisfied. We get the CLT theorem~\ref{thm:varying} with  $a_j = \nu_j  \tau \gamma_j$ and $b_j =  \tau \gamma_j + 1$.

The multiple Charlier ensemble  with $m=2$ has appeared before (although  not explicitly mentioned). Indeed, in \cite{Duits1}, the dynamics on interlacing particles systems of \cite{BF} with two speeds was studied.  By taking the marginal density at the horizontal level with $n$ particles, we obtain exactly the process described here. The different coefficients $\gamma_1$ and $\gamma_2$ represent two different speeds in the evolution of the particle system. The main result of \cite{Duits1} was the computation of the long time behavior of  global fluctuation for the two dimensional system, which turns out to be described by the Gaussian Free Field. Observe that the proof of the fluctuations of the present paper avoids many of the technical details that one needs to overcome in the approach of \cite{Duits1}. In fact, in an effort to keep the number of technical details to a minimum, \cite{Duits1} deals with only two possible speeds, whereas here we can allow $m$ different values without much extra effort.

\subsubsection{Multiple Krawtchouk Ensemble}

Let us now consider the case \eqref{eq:simpleRandomWalk}. Now note that
$$\frac{\chi_{i-1 \leq x\leq t+i-1}(x)}{(x-i+1)!(t-x+i-1)!}= (-1)^{n-1} \frac{(-x)_{i-1} (-t-n+1+x)_{n-i}}{x!(t+n-1-x)!},$$
for $x=0,\ldots, t+n-1$.
This implies that

$$P_t(x_i+1-j)= (p/(1-p))^{x_i} \frac{q_j(x_i)}{x_i! (t+n-1-x_i)!},$$
for $x_i=0,\ldots,t+n-1$, where $q_j$ is a polynomial of degree $n-1$. In fact, all $q_1,\ldots, q_n$ can be shown to be  linearly independent and thus, after some  standard rules for determinants, we find
$$\det\left( P_t(x_i+1-j)\right)_{i,j=1}^n = c_n(p,t) \prod_{i=1}^n \frac{(p/(1-p))^{x_i}}{x_i! (t+n-1-x_i)!} \det\left(x_i^{j-1}\right)_{i,j=1}^n,$$
where $c_n(p,t)$ is a constant independent of the $x_i$'s.

	Concluding, we find that the process \eqref{jpdfexample3} at time $t$ is indeed a MOPE on $\{0,1,\ldots,t+n-1\}$, with $\vec k = (k_1, \ldots, k_m)$ where $|\vec k|=n$, and
$$w_j(x)= \gamma_j^x, \quad j=1, \ldots,m, \qquad  \mu_t(x)= \frac{(p/(1-p))^{x}}{x!(t+n-1-x)!}.$$
These are the weights for the Multiple Krawtchouk Polynomials $K_{\vec k}^{\vec p,t+n-1}$ that were studied in \cite{ACV,HV,Ismail}. Here $\vec p=(p_1,\ldots, p_m)$ where $p_j \in (0,1)$ is the unique solution to $\gamma_j p/(1-p)=p_j/(1-p_j)$.  Naturally, when $m=1$, these reduce to the standard Krawtchouk polynomials. The nearest neighbor recurrence relation reads (see \cite[Section 3.4]{HV})
\begin{equation}\label{eq:krNNR}
x K_{\vec k}^{\vec p,t+n-1}(x)=  K_{\vec k+ \vec e_\ell} ^{\vec p,t+n-1}(x)+ b_{\vec k, \ell} K_{\vec k}^{\vec p,t+n-1}(x)+ \sum_{j=1}^ m a_{\vec k,j}K_{\vec k-\vec e_j}^{\vec p,t+n-1}(x).
\end{equation}	
where
$$ a_{\vec k,j}=p_j(1-p_j) k_j(t+n-|\vec k|),$$
and
$$b_{\vec k,j}=(t+n-1-|\vec k|) p_j+ \sum_{\ell=1}^m k_\ell  (1-p_\ell).$$ After a rescaling $t = \lfloor n \tau \rfloor $  (with $\lfloor q \rfloor$ denoting the integer part of $q$), $x=yn$, and $\tilde K_{\vec k}^{\vec p,\tau}(y)= n^{-|\vec k|}K_{\vec k}^{\vec p,[n \tau]+n-1}(x)$, ~\eqref{eq:krNNR} becomes
$$
 y \tilde K_{\vec k}^{\vec p,\tau}(y)=  \tilde K_{\vec k+ \vec e_\ell} ^{\vec p,\tau}(y)+ \frac{ b_{\vec k, \ell}}{n} \tilde K_{\vec k}^{\vec p,\tau}(y)+ \sum_{j=1}^ m \frac{a_{\vec k,j}}{n^2} \tilde K_{\vec k-\vec e_j}^{\vec p,\tau}(y).
$$
 Now it is easy to see that the conditions \eqref{eq:multipleNevaivarying} are satisfied and Theorem~\ref{thm:varying} applies with $a_j = p_j(1-p_j)\nu_j \tau$ and $b_j = \tau p_j + \sum_{r=1}^m \nu_r(1-p_r)$.

 \subsubsection{Multiple Meixner ensemble}\label{ss:Meixner}
 The last  example that we will consider is that of the multiple Meixner polynomials where consider the Markov chain with $P_t$ as in \eqref{eq:geometricRandomWalk}. We will only consider times $t\geq n$ and comment on $1 \leq t <n$ in the end of this paragraph.

 First note that, for $t \geq n$, we have
 $$
 	\begin{pmatrix}
 	t+x-i\\t-1
 	\end{pmatrix}= \frac{(t-n)!}{(t-1)!}\frac{(t-n+1)_x }{x!} (t-n+x+1)_{n-i} (x-i+2)_{i-1}
 $$
 Hence
 $$P_t(x+1-j)=\frac{(t-n+1)_x  a^x}{x!} q_j(x),$$
 where $q_j$ is a polynomial of degree $n-1$. And therefore (again after matrix manipulations) we find
 $$\det\left( P_t(x_i+1-j)\right)_{i,j=1}^n = c_n(a,t) \prod_{i=1}^n \frac{(t-n+1)_{x_i}  a^{x_i}}{x_i!} \det\left(x_i^{j-1}\right)_{i,j=1}^n,$$
 for some $c_n(a,t)$ independent of the $x_i$'s.

 Concluding, we find that the process \eqref{jpdfexample3} at time $t$ is indeed a MOPE on $\{0,1,\ldots\}$, with $\vec k= (k_1, \ldots, k_m)$, and
 $$w_j(x)=\gamma_j^x, \qquad j=1,\ldots, m, \qquad\mu_t(x)= \frac{(t-n+1)_x a^x}{x!}.$$
 These are precisely the weights for the multiple Meixner polynomials $M^{(1)}_{\vec k}(x)$ of the first kind (there are two families of multiple Meixner polynomials). They satisfy the recurrence relations (\cite[Section 3.3]{HV})
\begin{equation}\label{eq:krNNR}
x M^{(1)}_{\vec k}(x)=  M_{\vec k+ \vec e_\ell} ^{(1)}(x)+ b_{\vec k, \ell} M_{\vec k}^{(1)}(x)+ \sum_{j=1}^ m a_{\vec k,j}M_{\vec k-\vec e_j}^{(1)}(x).
\end{equation}	
where
$$ a_{\vec k,j}= (t-n+|\vec k|)\frac{k_j a \gamma_j}{(1-a \gamma_j)^2},$$
and
$$b_{\vec k,j}=(t-n+1+|\vec k|)\frac{a \gamma_j}{1-a \gamma_j}+\sum_{i=1}^m \frac{k_i}{1-a\gamma_i}.$$

To get a reasonable limit we set $t=n\tau$ and scale $x=n \xi$. After setting $\tilde M^{(1)}_{\vec k}(\xi)=n^{-|\vec k|}  M_{\vec k}^{(1)}(n \xi)$ we obtain the recurrence
 \begin{equation}\label{eq:krNNR}
 \xi \tilde M^{(1)}_{\vec k}(\xi)=  \tilde  M_{\vec k+ \vec e_\ell} ^{(1)}(\xi)+ \frac{ b_{\vec k, \ell}}{n} \tilde M_{\vec k}^{(1)}(\xi)+ \sum_{j=1}^ m \frac{a_{\vec k,j}}{n^2}\tilde M_{\vec k-\vec e_j}^{(1)}(\xi).
 \end{equation}	
   Now it is easy to see that the conditions \eqref{eq:multipleNevaivarying} are satisfied and Theorem~\ref{thm:varying} applies with $$a_j = \frac{\tau \nu_j a \gamma_j }{(1-a\gamma_j)^2 }\quad \text{ and } \quad  b_j = \frac{\tau a \gamma_j}{1-a\gamma_j} + \sum_{i=1}^m \frac{\nu_i}{1-a\gamma_i}.$$

We end this example by commenting on the condition that $t \geq n$. Note that this condition does not enter in the previous examples on multiple Charlier and multiple Krawtchouk polynomials. The reason is not merely technical but has an interpretation. Recall that our state space is the Weyl chamber and we have $n$ processes that are non-colliding. One can verify by evaluating the determinant in  \eqref{jpdfexample3} with \eqref{eq:geometricRandomWalk} explicitly, that if at time $t$ the $j$-th process (counted from below) has location $x$ and the $(j+1)$-th process has location $x+1$, then the $j$-th process  will be blocked by the $(j+1)$-th process, i.e., will not be able to jump at time $t+1$. Therefore it will remain at location $x$ at time $t+1$. Since the processes start at $n$ consecutive initial points at $t=0$, the lowest process will only be able to move starting from $t=n$. At time $t\leq n$ only the top $t$ processes 
 have been able to move. This phenomenon does not occur in the previous two examples  where each process is able to move at all times $t >0$.

 Finally, note that for $t=n$, we have $\mu(x)=a^x$ and the MOPE is exactly the last example from the Introduction.

\subsubsection{Multi-time fluctuations}
In the three examples, we have only considered the distribution of the positions at a fixed time. It is also possible to look at the joint distribution at several times.

Let
$0=t_0< t_1 <t_2< \ldots <t_N$ and set $\Delta_r=t_r-t_{r-1}$ for $r=1,\ldots, N$. Denote the positions of the processes at time $t_r$ by $x_j^r$ for $r=1,\ldots,N$ and order them according to $x_1^r<\ldots <x_n^r$. Then the probability of having the processes go through the points $\{x_j^r\}_{j,r=1}^{n,N}$ at times $t_1,\ldots, t_m$ is proportional to
\begin{equation} \label{jpdfexample4}
 \det \left( P_{\Delta_1}(x_i^1-\xi_j)\right)_{i,j=1}^n \left(\prod_{r=2}^{N} \det \left(P_{\Delta_{r}}(x_i^{r}-x_j^{r-1})\right)_{i,j=1}^n \right)\frac{ \det\left( \alpha_j^{x_i^{N}}\right)_{i,j=1}^n }{\det\left( \alpha_j^{\xi_i}\right)_{i,j=1}^n},
\end{equation}
giving a product of $N+1$ terms. This raises the interesting question whether we can describe the multi-time fluctuations using multiple orthogonal polynomials. In \cite{Duits2} this was done for regular orthogonal polynomials (that is, $m=1$). In that paper it was proved that the multi-time fluctuations are governed by the two-dimensional Gaussian Free Field with Dirichlet boundary conditions using properties of the orthogonal polynomials and the two-dimensional extension of the methods introduced in \cite{BD}. We expect a similar statement to be true in our setting, perhaps by a similar argument as in \cite{Duits2}. We will not address this very interesting question in this paper,  but do intend to return to this topic in future work.

\subsubsection{Schur measure and process} \label{sec:Schur}
For completeness we describe how the dynamics of above is a special example of the Schur process and the MOPE are special examples of the Schur measure. 

Let $a=(a_1,a_2,\ldots) \in \mathbb R_{\geq 0}^\infty$, $b=(b_1,b_2,\ldots) \in \mathbb R_{\geq 0}^\infty$  and $c\geq 0$. Assume that $0 \leq a_j <1$ and 
$$\sum_{j=1}^\infty (a_j +b_j) < \infty.$$
Then, for $|z|<1$, define 
\begin{equation}\label{eq:Hz}
	H(z;a,b,c)=e^{cz} \prod_{j=1}^\infty\frac{1+b_jz}{1-a_jz} 
	\end{equation} and set
$$ H(z;a,b,c)=\sum_{k=0}^\infty h_k(a,b,c) z^k
$$
We also set $h_k=0$ for $k\leq 0$. Let $\lambda=(\lambda_1,\lambda_2,\ldots)$ be a partition. That is, $\lambda_j\in \mathbb Z_{\geq 0}$ and $\lambda_i\geq \lambda_j$ whenever $i \leq j$. Denote the length of the partition by $\ell(\lambda)=\max \{k \mid \lambda_k >0\}$ (we will always assume $\ell(\lambda)< \infty$).  Then the Schur function is defined by 

\begin{equation} \label{eq:JacobiTrudi1}
s_\lambda(a,b,c)=\det\left( h_{\lambda_j-j+i}(a,b,c)\right)_{ i,j=1}^{\ell(\lambda)}
\end{equation}
This is the Jacobi-Trudi formula for Schur functions.  If $b=0,c=0$ and $a=(a_1,\ldots,a_n,0,0,\ldots)$ then it can be shown that this equals the Schur polynomial  defined in  \eqref{eq:Schurpol} for $\ell(\lambda)\leq n$  and $0$ for $\ell(\lambda)>n.$

For two partitions $\lambda$ and $\mu$ we say that $\lambda \geq  \mu$ if and only if $\lambda_j \geq \mu_j$. The skew Schur function is defined as  

\begin{equation} \label{eq:JacobiTrudi2}
s_{\lambda / \mu}(a,b,c)= \det \left(h_{\lambda_j-j-\mu_i+ i}(a,b,c)\right)_{i,j=1}^{\ell(\lambda)},
\end{equation}
for $\lambda \geq \mu$ and $s_{\lambda / \mu}(a,b,c)=0$ otherwise. Note that if $\mu$ is the trivial partition (i.e. $\mu_j=0$) then $s_{ \lambda/\mu}=s_\lambda$.  Also, $s_{\lambda/ \mu}(0,0,0)=1$ if $\lambda=\mu$ and $s_{\lambda/ \mu}(0,0,0)=0$ otherwise.

To be precise, our definitions \eqref{eq:JacobiTrudi1} and \eqref{eq:JacobiTrudi2} are specializations of (skew) Schur functions. The (skew) Schur functions are elements of the algebra  $\Lambda$ of symmetric functions in infinitely many variables (for a proper definition of this algebra see \cite{BG2}).
The definitions above are the result of applying different  specializations to that algebra (a specialization is an algebra homomorphism between $\Lambda$ and $\mathbb C$). In fact, these are known as the Schur positive specializations since these are precisely all specializations such that $s_\lambda \geq 0$ for any~$\lambda.$ 

The Schur measure, introduced by Okounkov in \cite{OK}, is the probability measure on partitions $\lambda$ given by 
\begin{equation} \label{eq:Schurmeasure}
\mathcal P(\lambda)=\frac{1}{Z} s_{\lambda}(\rho_+) s_{\lambda}(\rho_-),
\end{equation}
where $Z_n$  is a normalizing constant and $\rho_\pm$ are of the form   $\rho_\pm=(a_\pm,b_\pm,c_\pm)\in \mathbb R_{\geq 0}^\infty \times  \mathbb R_{\geq 0}^\infty \times \mathbb R_{\geq 0}$.  Note that \eqref{eq:preSchurmeasure} is a special case for an appriate choice of $\rho_+$ and $\rho_-$. The Schur process,  introduced by Okounkov and Reshitikhin in \cite{OR},  is the probability measure on a pair of sequences of partitions $\{\lambda^{(j)}\}_{j=1}^N$ and 
$\{\mu^{(j)}\}_{j=1}^{N-1}$ such that $\lambda^{(j)} \geq \mu^{(j)} $ and $\lambda^{(j+1)}\geq \mu^{(j)}$ for $j=1,\ldots, N-1$ proportional to 
\begin{equation} \label{eq:Schurprocess}
s_{\lambda^{(1)}}(\rho^{(1)}_+) \left(  
\prod_{j=2}^{N}	s_{ \lambda^{(j-1)}/\mu^{(j-1)}}(\rho^{(j-1)}_-) 	s_{ \lambda^{(j)}/\mu^{(j-1)}}(\rho^{(j)}_+) \right) s_{\lambda^{(N)}}(\rho^{(N)}_-)
\end{equation}
where $\rho_\pm^{(j)}=(a_\pm^{(j)},b_\pm^{(j)},c_\pm^{(j)})\in \mathbb R_{\geq 0}^\infty \times  \mathbb R_{\geq 0}^\infty \times \mathbb R_{\geq 0}$ for $j=1,\ldots,N$. It follows from the Cauchy-Binet identity that the marginal densities  for each $\lambda^{(j)}$ or $\mu^{(j)}$ in the Schur process are Schur measures. For future reference, we recall that if $\rho_-^{(j)}=(0,0,0)$  then $\mu^{(j)}=\lambda^{(j)}$, for $j=1,\ldots , N-1$ (with probability one).

Coming back to the Markov chain, we show first that the transition probabilities can be rewritten using Schur function where the parameters depend on the choice \eqref{eq:Poisson}, \eqref{eq:simpleRandomWalk} or \eqref{eq:geometricRandomWalk}.
\begin{lemma}\label{lem:connectiontoschur}
	Set 
	$\lambda_j=x_{n-j+1}-n+j$ and $\mu_j=y_{n-j+1}-n+j$ for $j=1,\ldots, n$. \\
	Then, with $P_t$ as in \eqref{eq:Poisson}, we have
	$$\det \left(P_t(x_i-y_j)\right)_{i,j=1}^n =d_t s_{\lambda/\mu }(0,0,t),$$
	and $d_t$ is some constant. \\
	With $P_t$ as in \eqref{eq:simpleRandomWalk}, we have
	$$\det \left(P_t(x_i-y_j)\right)_{i,j=1}^n =d_t s_{\lambda/\mu }(0,b^t,0)$$
	where $b^t=\left(\underset{t}{\underbrace{\frac{p}{1-p},\ldots,\frac{p}{1-p}}},0,0,\ldots\right)$  and $d_t$ is some constant.  \\
	With $P_t$ as in \eqref{eq:geometricRandomWalk}, we have
	$$\det \left(P_t(x_i-y_j)\right)_{i,j=1}^n =d_t s_{\lambda/\mu }(a^t,0,0)$$
	where $a^t=(\underset{t}{\underbrace{a,\ldots,a}},0,0,\ldots)$ and $d_t$ is some constant.  
\end{lemma}
\begin{proof}
	First note that in the new variables (after flipping the order of the rows of columns)
	$$\det \left(P_t(x_i-y_j)\right)_{i,j=1}^n= \det \left(P_t(\lambda_{i}-i-\mu_j+j)\right)_{i,j=1}^n$$
	By the Jacobi-Trudi formula \eqref{eq:JacobiTrudi2}, we  thus need to show that $$P_t(k)=c_{t,a,b,c}h_k(a,b,c),$$ where $a,b$ and $c$ are as indicated in the lemma for the three different situations and $c_{t,a,b,c}$ is a constant independent of $k$.  That means we need to verify that 
	$$\sum_{k=0}^\infty P_t(k)z^k=c_{t,a,b,c} H(z;a,b,c).$$
	In each of the three cases these are some elementary Taylor series leading to 
	$$\sum_{k=0}^\infty P_t(k)z^k= \begin{cases}
	e^{t(z-1)}, & \text{ for \eqref{eq:Poisson}} \\
	(1-p+p z)^t, & \text{ for  \eqref{eq:simpleRandomWalk}} \\
	\left(\frac{1-a}{1-az}\right)^t, & \text{ for \eqref{eq:geometricRandomWalk}}.\end{cases}
	$$
	Comparing this with \eqref{eq:Hz} gives the statement.
\end{proof}
Using this lemma we see that probability measure \eqref{jpdfexample4} can be written as a product of Schur functions (with properly chosen parameters). Now note that in the change of variables the initial condition $\xi_i=i-1$ is mapped to the empty partition. Thus, by Lemma \ref{lem:connectiontoschur} and \eqref{eq:Schurpol} we see that we can write the multi-time probability function  \eqref{jpdfexample4} as a probability density function on sequences of partitions $\lambda^{(1)}  \leq \lambda^{(2)} \leq \ldots \leq \lambda^{(N)}$, where $\lambda^{(r)}_j=x^r_{n-j+1}-n+j$,
proportional to 
$$s_{\lambda^{(1)}} (\rho_{\Delta_1})\left(\prod_{r=2}^N s_{\lambda^{(r)}/\lambda^{(r-1)}} (\rho_{\Delta_r})\right) s_{\lambda^{(N)}}(\alpha,0,0).$$
where $ \alpha=(\alpha_1,\alpha_2,\ldots,\alpha_N,0,\ldots)$ and $\rho_{\Delta_r}$ are the parameters indicated in Lemma \ref{lem:connectiontoschur}.  By comparing this to \eqref{eq:Schurprocess} we see that  this is a special case of the Schur process  where $\rho_+^{(j)}=\rho_{\Delta_j}$ for $j=1, \ldots, N$,  $\rho_-^{(j)}=(0,0,0)$ (and hence $\mu^{(j)}=\lambda^{(j)}$) for $j=1,\ldots, N-1$, and $\rho_-^{(N)}= (\alpha,0,0)$. For $N=1$  the marginal distribution of $\lambda^{(1)}$ is proportional to
$$s_{\lambda^{(1)}} (\rho_{\Delta_1}) s_{\lambda^{(1)}}(\alpha,0,0), $$  
and is a special case of Schur measure \cite{OK}. 
We recall that the three classical discrete MOPE's discussed above arise as the fixed time marginal distribution when  setting $\alpha_j \to q_k$ for $q_1,\ldots,q_m$ with corresponding multiplicities $k_1,\ldots, k_m$, which shows that these MOPE's are particular specializations of the Schur measure.

\end{document}